\documentclass[pdflatex,sn-mathphys-num]{sn-jnl}

\usepackage{graphicx}%
\usepackage{multirow}%
\usepackage{amsmath,amssymb,amsfonts}%
\usepackage{amsthm}%
\usepackage{mathrsfs}%
\usepackage[title]{appendix}%
\usepackage{xcolor}%
\usepackage{textcomp}%
\usepackage{manyfoot}%
\usepackage{booktabs}%
\usepackage{algorithm}%
\usepackage{algorithmicx}%
\usepackage{algpseudocode}%
\usepackage{listings}%
\usepackage{epsfig}
\usepackage[hang]{caption2}
\usepackage[normalsize,it,hang]{subfigure}
\newcommand\tailleFFF{0.8}
\usepackage{siunitx}
\usepackage{here}

\usepackage{lscape}

\theoremstyle{thmstyleone}%
\newtheorem{theorem}{Theorem}
\newtheorem{proposition}[theorem]{Proposition}%

\newtheorem{corollary}[theorem]{Corollary}
\newtheorem{lemma}[theorem]{Lemma}

\theoremstyle{thmstyletwo}%
\newtheorem{remark}{Remark}%

\theoremstyle{thmstylethree}%
\newtheorem{definition}{Definition}%

\begin{document}

\title[Time-Invariant Linear Quadratic Regulators:  A New Look]{{\small {\tt Linear Algebra and its Applications,} {\it to appear}} \\Time-Invariant Linear Quadratic Regulators: \\ A New Look} 


\author[1]{\fnm{Cédric} \sur{Join}}\email{cedric.join@univ-lorraine.fr}

\author[2]{\fnm{Emmanuel} \sur{Delaleau}}\email{emmanuel.delaleau@enib.fr}

\author*[3,4]{\fnm{Michel} \sur{Fliess}}\email{michel.fliess@sorbonne-universite.fr, michel.fliess@alien-sas.fr}

\affil[1]{\orgdiv{CRAN (CNRS UMR 7039)}, \orgname{Universit\'{e} de Lorraine}, \orgaddress{\street{BP 239}, \\ \postcode{54506} \city{Vand{\oe}uvre-l\`{e}s-Nancy}, \country{France}}}

\affil[2]{%
    \orgdiv{IRDL (CNRS UMR 6027)}, %
    \orgname{%
    Bretagne INP, 
    ENSTA, \\ Institut Polytechnique de Paris, Univ. Brest, 
    Univ. Bretagne-Sud}, \\ %
    \orgaddress{%
    \street{445 avenue du Technop\^{o}le}, \postcode{29280} %
    \city{Plouzané}, %
    \country{France}.
    }
}


\affil[3]{\orgdiv{LJLL (CNRS UMR 7598)}, \orgname{Sorbonne Universit\'{e}}, \orgaddress{\street{4 place Jussieu}, \newline \postcode{75005} \city{Paris}, \country{France}}}

\affil[4]{\orgdiv{AL.I.E.N.}, 
\orgaddress{\street{7 rue Maurice Barr\`{e}s}, \postcode{54330} \city{V\'{e}zelise}, \country{France}}}

\maketitle
\newpage
\section*{Abstract}

This paper demonstrates that linear ordinary differential equations with constant coefficients serve as the main building block for time-invariant linear quadratic regulators (LQRs). By applying the Euler-Lagrange equations to a flat output—i.e., to a basis of the free finitely-generated module over the commutative principal ideal ring of time-invariant linear differential operators with respect to the time derivation—this result emerges directly. In module-theoretic terms, the freeness of this structure corresponds to the classical Kalman controllability condition
(M. Fliess, Some basic structural properties of generalized linear systems, Systems \& Control Letters, 15 (1990) 391-396).
Within this framework, the two-point boundary value problem, including the control variables, becomes analytically tractable, enabling the derivation of optimal time horizons, optimal parameter design, and rest-to-rest trajectories. Furthermore, a novel turnpike-like phenomenon is established via nonstandard analysis. The feedback loop is closed thanks to the HEOL approach
(C. Join, E. Delaleau, M. Fliess, Flatness-based control revisited: The HEOL setting, Comptes Rendus Math\'{e}matique, 362 (2024) 1693-1706), an extension of model-free control (MFC) that leverages intelligent controllers. Our viewpoint, which significantly simplifies the presentation by avoiding the conventional Kalman state-space representation, relies on foundational algebraic concepts that, while elementary, are rarely used in today's control theory. To aid intuition, some rudimentary examples are presented, most of which include computer simulations.

\section*{Keywords}
LQR, controllability, observability, two-point boundary value problem, optimal time horizon, optimal parameter design, optimal-rest-to-rest trajectory, turnpike, HEOL, homeostat, module theory, differential algebra, nonstandard analysis.




\newpage 

\section*{Highlights}
\begin{itemize}
    \item By leveraging the equivalence between Kalman's controllability of a time-invariant linear system and the freeness of the associated module, the Euler-Lagrange equation for a linear quadratic regulator reduces to the simplest form of ODEs: a time-invariant linear differential equation with respect to a flat output, i.e., a basis of the free module. This property extends to any polynomial Lagrangian of degree 2 in the components of the flat output and their derivatives up to a finite order.
\item In contrast to Kalman's classical approach, this framework enables the explicit resolution of the two-point boundary value problem, with the Wronskian matrix serving as the central tool. Several applications are presented:
\begin{enumerate}
\item Time optimality: Unlike classical LQR, the criterion can be optimized with respect to the time horizon.
\item Parameter optimization: System parameters can be selected to optimize the performance criterion.
\item Control specifications: By incorporating higher-order derivatives of a flat output, it becomes possible to specify initial and final values of control variables, enabling smooth starts and stops (starting and stopping at rest).
\item Turnpike property: A novel turnpike-like behavior is derived using elementary nonstandard analysis. This property characterizes the phenomenon whereby optimal control solutions remain close to a steady state when the time horizon is very large.
\end{enumerate}

\item Rather than employing Kalman's setting based on the Riccati differential equation, this article proposes closing the loop via the new HEOL setting, which is a tracking controller derived from intelligent controllers inspired by model-free control (MFC). This viewpoint offers enhanced robustness against uncertainties and disturbances.
\end{itemize}

\newpage

\section{Introduction}\label{sec1}
\subsection{Historical and conceptual background: some hints}
The Kalman {\em linear quadratic regulator} (\emph{LQR}) has played a prominent r\^{o}le in control theory and engineering for more than 65 years (see, e.g., among the numerous monographs and textbooks, \cite{abou,anderson,andrea,ar,athans,brockett,bryson,follinger,friedland,hardt,kailath,kw,liberzon,lewis,lunze,mehrmann,sontag,trelat05,trelat24,wolovich}). In his groundbreaking paper \cite{mex}, Kalman wrote: ``The principal contribution of the paper lies in the introduction and exploitation of the concepts of \emph{controllability} and \emph{observability}.'' These two fundamental notions were introduced by Kalman \cite{moscou} in his famous plenary conference at the first IFAC Congress (Moscow, 1960). LQR, controllability, and observability are taught nowadays in the Kalman manner, i.e., by defining them through the Kalman state-variable representation. In 1990, controllability and observability were defined \cite{fliess90} more intrinsically, i.e., independently of the above setting, by tools stemming from abstract linear algebra, i.e., module theory: 
\begin{itemize}
    \item A time-invariant linear system becomes a finitely generated module over the commutative principal ideal ring $\mathbb{R}\left[\frac{d}{dt}\right]$ of linear differential operators $\sum_{\rm finite} a_\kappa \frac{d^\kappa}{dt^\kappa}$, $a_\kappa \in \mathbb{R}$, i.e., a most elementary object (see, e.g., \cite{jacobson,lang}). 
    \item Controllability and module freeness are equivalent: there exists a {\em flat output},\footnote{See \cite{flmr} for the origin of this terminology.} i.e., a basis of the module such that:
    \begin{itemize}
        \item its components and their derivatives of arbitrary orders are $\mathbb{R}$-linearly independent;
        \item any system variable is a $\mathbb{R}$-linear combination of those components and their derivatives up to some finite order.
    \end{itemize}
    \item Observability means that the control and output variables span the system module: any system variable is a $\mathbb{R}$-linear combination of the control and output variables and their derivatives up to some finite order.
\end{itemize}
\begin{remark}\label{rem}
The above formalism, which ought to be compared with Rosenbrock's system matrix \cite{rosenbrock}, already permitted to clarify poles and zeros of multivariable systems \cite{poles}, parameter identification \cite{esaim,springer}, system interconnections \cite{inter}, state estimation \cite{cr}. See also \cite{bourles}.
\end{remark}

\subsection{The new LQR}
\subsubsection{Generalities}
In the Kalman LQR, the Lagrangian, or cost function, $\mathcal{L}$ is a quadratic form in the state and control variables. Controllability in our module-theoretic framework implies that each of those variables may be expressed as a $\mathbb{R}$-linear combination of the flat output and their derivatives up to some finite order. The Lagrangian $\mathcal{L}$ may thus be seen as a polynomial of degree $2$ with respect to the components of a flat output and their derivatives up to some finite order. The integral criterion $J = \int_0^T \mathcal{L}dt$, where $T > 0$ is the \emph{time horizon}, yields the Euler-Lagrange equation, i.e., the well-known equation from the calculus of variations. 
The following points, where the last four are inspired by \cite{reinhard}, \cite{zuazua}, \cite{boscain}, \cite{trelat23}, show a major departure from the Kalman LQR:
\begin{enumerate}
    \item The Euler-Lagrange equation becomes an elementary linear differential equation with constant coefficients.
    \item The two-point boundary value problem is tractable. 
    \item The criterion $J$ becomes a function of the time horizon $T$ when initial and terminal conditions are fixed. For an extremum of $J$, $T$ is said to be {\em optimal}. 
    \item The \emph{optimal parameter design} selects system parameters to optimize the criterion.
     \item $\mathcal{L}$ is not necessarily limited to a quadratic form with respect to state and control variables. By taking, for instance, higher order derivatives of the flat outputs, it is possible to choose the initial and terminal values of the control variables and, for example, to start and stop at rest.\footnote{Stopping at rest is obviously related to \emph{finite-time stability} (see, e.g., \cite{amato}). Compare with the stabilization via an infinite time horizon in the Kalman LQR.}
    \item The \emph{turnpike phenomenon} (see, e.g., \cite{trelat23}, and references therein) is approached with Robinson's \emph{nonstandard analysis} \cite{robinson}, which has already been employed in practice (see, e.g., \cite{lobry} and references therein).
\end{enumerate}

\subsubsection{Some illustrative calculations}
\paragraph{1) Our LQR}
Consider 
\begin{equation}\label{cal}
\dot{y} + 3y = 2u 
\end{equation}
where $u$ (resp. $y$) is the control (resp. output) variable. It is obvious that $y$ is a flat output. Introduce the integral criterion 
\begin{equation}\label{crit}
J = \int_{0}^{1} \left(u^2 (\tau)+y^2 (\tau)\right) d\tau =\frac{1}{4} \int_{0}^{1} \left(13 y^2(\tau) + 6 y(\tau)\dot{y}(\tau) + \dot{y}^2 (\tau)\right)d\tau
\end{equation}
The corresponding Euler-Lagrange equation reads $\ddot{y} - 13y =0$. Its general solution is $y = c_1 \exp (t \sqrt{13} + c_2 \exp (- t \sqrt{13}$, $c_1, c_2 \in \mathbb{R}$. With $y(0) = 0$, $y(1) = 1$, $c_1 = - c_2 \approx 0.027$. With the criterion
$$\mathcal{J} = \int_{0}^{1} (u^2 + (1 - y)^2)d\tau = \frac{1}{4} \int_{0}^{1} (13y^2 - 8y + 6y\dot{y} + \dot{y}^2 + 4)d\tau$$ the Euler-Lagrange equation becomes $\ddot{y} - 13y +4 = 0$. Its general solution is $y = \frac{4}{13} + c_1 \exp (t \sqrt{13}) + c_2 \exp (- t \sqrt{13})$. With again $y(0) = 0$, $y(1) = 1$, $c_1 \approx -0.327$, $c_2 \approx 0.019$. With a free endpoint, the corresponding boundary condition is (see, e.g., \cite{liberzon}, Sec. 2.3.5, pp. 42-44) $3y(1) + \dot{y}(1) = 0$. It yields 
$c_1 \approx - 0.327$, $c_2 \approx 0.019$. See the results in Fig. \ref{O}.

\paragraph{2) The Kalman LQR}
Take again Eq. \eqref{cal}. Here, the criterion, in a tracking version of the Kalman LQR (see, e.g., \cite{lewis}, chap. 4), is adapted from Eq. \eqref{crit}: $\mathbf{J} =  S_f(y(1) - 1)^2 + \mathcal{J}$, where $S_f \in \mathbb{R}$  is a constant. It yields the control
$u = 2(\lambda -K_p y)$, where the gain $K_p$ and the costate $\lambda$ respectively satisfy a Riccati differential equation $\dot K_p = 6K_p  + 4K_p ^2-1$ and a linear differential equation $\dot \lambda =(3+4K)\lambda - 1$, which must be integrated backwards in time, with $K_p(1) = \lambda(1) = 1$. Besides the greater numerical complexity of this approach compared to our LQR,  Fig. \ref{OL} shows a poor tracking performance if $S_f$ is not ``very large.''

\begin{figure}[H]
\centering
\subfigure[\footnotesize $u^\star$]
{\epsfig{figure=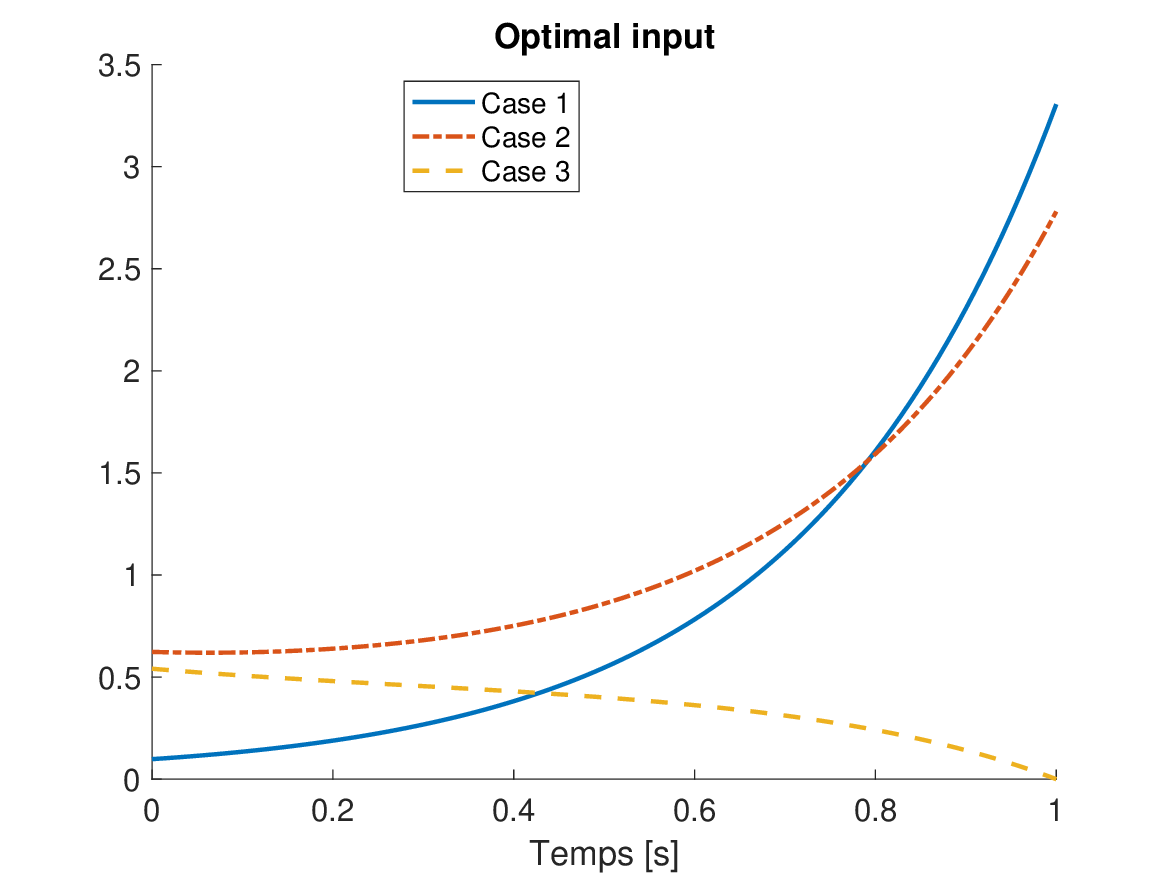,width=0.7\textwidth}}
\\
\subfigure[\footnotesize $y^\star$]
{\epsfig{figure=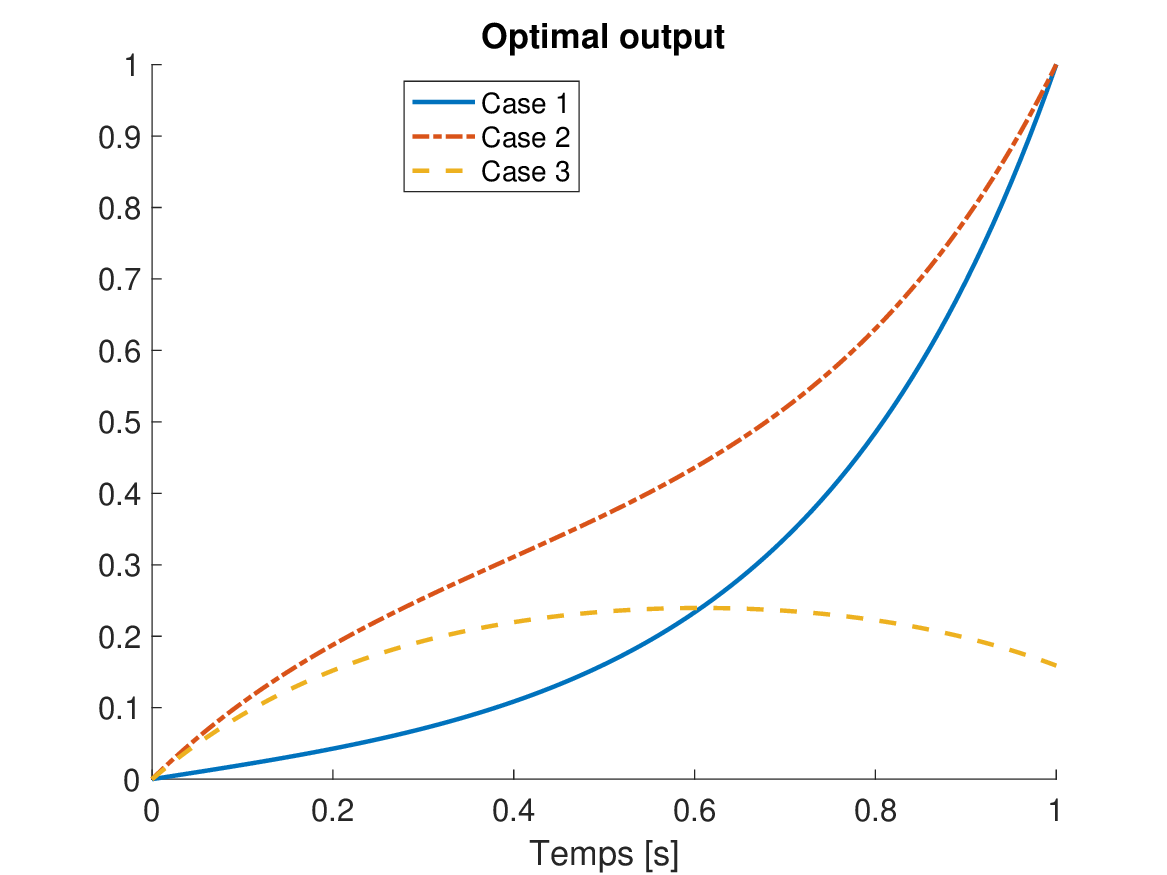,width=0.7\textwidth}}
\caption{Comparisons for the new LQR}\label{O}
\end{figure}

\begin{figure}[H]
\centering
\subfigure[\footnotesize $u^\star$]
{\epsfig{figure=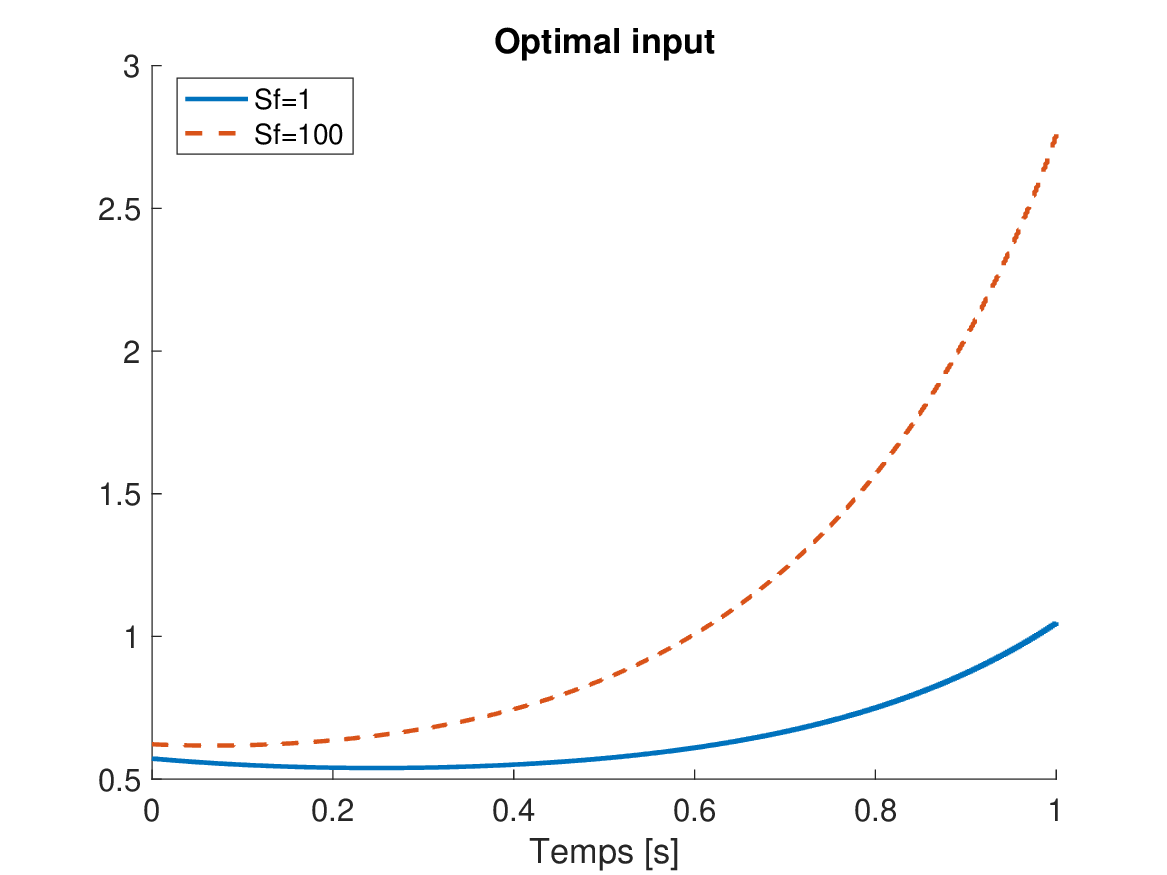,width=0.7\textwidth}}
\\
\subfigure[\footnotesize $y^\star$]
{\epsfig{figure=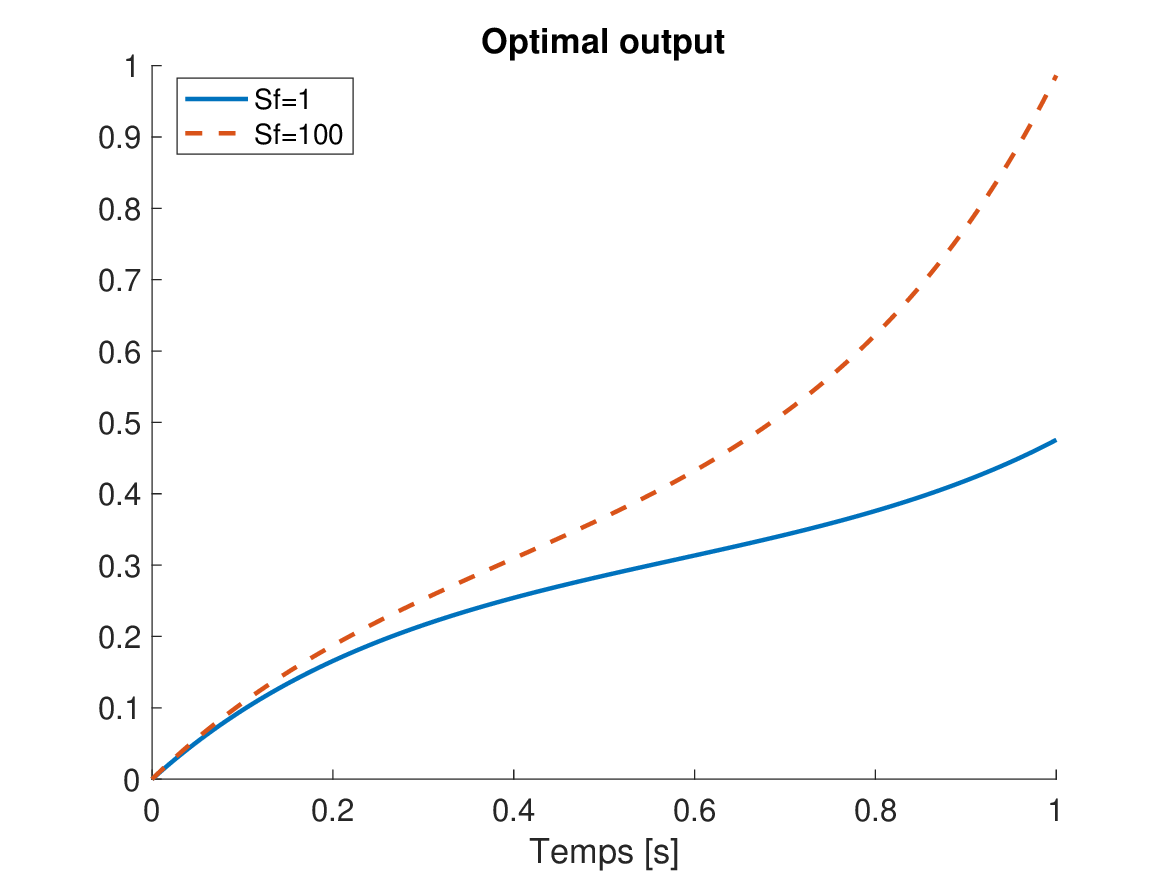,width=0.7\textwidth}}
\caption{The Kalman LQR}\label{OL}
\end{figure}

\subsubsection{HEOL}

What remains, however, is to enter into true optimal control and not to remain in the world of the calculus of variations. In other words, we must complete the above results with a feedback loop. Let $u^\star$ be the nominal control corresponding to the reference trajectory given by the Euler-Lagrange equation. Set $u = u^\star + \Delta u$, where $\Delta u$ is defined by the \emph{intelligent} controller associated with the recent HEOL setting \cite{heol} which is based on the variational system.\footnote{See \cite{degorre,wien,mfpc,maroc} for first concrete illustrations with respect to nonlinear systems.}
Then, the output stays close to the reference despite model mismatches and disturbances. Such tracking controllers are inspired by {\em model-free control} (\emph{MFC}) \cite{mfc1,mfc2}, which has numerous concrete applications.

Our paper is organized as follows. The module-theoretic framework of Sect. \ref{1} permits us to introduce, more or less as in \cite{fliess90,transf}, time-invariant linear systems,\footnote{Time-varying linear systems are considered in \cite{fliess90,transf}.} the state-variable representation, controllability, observability, transfer functions, and matrices, and for the first time, 
\begin{itemize}
    \item \emph{well-formed} dynamics (Sect. \ref{wellformed}) where filtrations and graduations are important for understanding their structural properties;
    \item the rank of the presentation matrix (Sect. \ref{pres});
    \item {\em Lagrangians}, or \emph{cost functions}, (Sect \ref{lagr}) via the symmetric algebra of the system module; 
    \item \emph{variational systems} (Sect. \ref{variat}), which are straightforward adaptations of \cite{johnson}, where K\"{a}hler differentials (see, e.g., \cite{eisenbud}) are extended to differential algebra.\footnote{See \cite{flmr,heol} for nonlinear control systems.}
\end{itemize}
Sect. \ref{euler} employs, once again through modules, the Euler-Lagrange equations to derive open-loop optimal solutions. Solving the two-point boundary value problem is accomplished using the Wronskian (see, e.g., \cite{bostan,kaplansky}) with coefficients in a differential ring of entire analytic functions. Four illustrative examples are examined in Sect. \ref{4}. 
Closing the loop in Sect. \ref{closing} is achieved via the variational system of Sect. \ref{variat} by two means:
 \begin{enumerate}
     \item a classic static state feedback for pole placement, 
     \item the HEOL setting with a numerical toy example: it will undoubtedly prove to be more effective.
 \end{enumerate}
Sect. \ref{conclusion} concludes with some remarks on future research directions and on our choice of the algebraic formalism.

\section{Module-theoretic preliminaries\protect\footnote{See, e.g., \cite{bourbaki,eisenbud,jacobson,kaplansky,lang,shaf} for basics in algebra. The wiki \emph{nLab} might also be helpful.}}\label{1}

\subsection{Rings and modules}\label{ringmodule}

The base field~$k$ is any commutative field. The set $k\left[\frac{d}{dt}\right]$ of $k$-linear differential operators $$\sum_{\rm finite}a_\alpha\frac{d^\alpha}{dt^\alpha}, \quad a_\alpha \in k$$ is a principal ideal commutative ring.
Let us restrict ourselves to the abelian category {\bf Mod} of finitely generated $k\left[\frac{d}{dt}\right]$-modules, without any loss of generality for time-invariant, finite-dimensional linear control systems. Let $\mathfrak{M}$ be such a module: 
\begin{itemize}
    \item An element $\mathfrak{m} \in \mathfrak{M}$ is said to be \emph{torsion} if and only if there exist $\pi \in k\left[\frac{d}{dt}\right]$, $\pi \neq 0$, such that $\pi \mathfrak{m} = 0$. 
    \item The set $\mathfrak{M}_{\rm tor}$ of all torsion elements in $\mathfrak{M}$ is a submodule. 
    \item $\mathfrak{M}$ is said to be \emph{torsion} (resp. \emph{torsion-free}) if and only if $\mathfrak{M} = \mathfrak{M}_{\rm tor}$ (resp. $\mathfrak{M}_{\rm tor} = \{0\}$).
    \item Any torsion-free module is \emph{free}:
    \begin{itemize}
        \item a module is free if and only if there exists a \emph{basis} $\{b_\iota \mid \iota \in I\}$; 
        \item the elements of the basis are linearly independent over $k\left[\frac{d}{dt}\right]$;
        \item $I$ is finite, and two bases have the same cardinality, called its \emph{rank}, or sometimes its \emph{dimension}.
    \end{itemize}
    \item $\mathfrak{M} \simeq \mathfrak{M}_{\rm tor} \bigoplus \Phi$, where $\Phi$, which is isomorphic to $M/M_{\rm tor}$, is free:
    \begin{itemize}
        \item the rank of $\mathfrak{M}$, written ${\rm rk} (M)$, is the rank of $\Phi$;
        \item the rank of $\mathfrak{M}$ is $0$ if and only if $\mathfrak{M}$ is torsion.
    \end{itemize}
    \item  The dimension of $\mathfrak{M}$ as a $k$-vector space is finite if and only if $\mathfrak{M}$ is torsion.    
    \end{itemize}
    
\noindent {\bf Notation.} For any subset ${\bf w} = \{w_\iota\mid \iota \in \mathcal{I} \} \subset \mathfrak{M}$, write ${\rm span}_{k\left[\frac{d}{dt}\right]}({\bf w})$ (resp. ${\rm span}_k({\bf w})$)  the $k\left[\frac{d}{dt}\right]$-submodule (resp. the $k$-vector subspace) of $\mathfrak{M}$  generated by ${\bf w}$.

\subsection{Systems, dynamics, control and output variables}

A \emph{linear system} is a module $\Lambda$. A \emph{linear dynamics} is a linear system $\Lambda$ equipped with a finite set ${\bf u} = \{u_1, \dots, u_m\}$ of \emph{control variables} such that the quotient module $\Lambda / {\rm span}_{k\left[\frac{d}{dt}\right]}({\bf u})$ is torsion. The control variables are said to be \emph{independent} if and only if 
\begin{itemize}
    \item ${\rm span}_{k\left[\frac{d}{dt}\right]}({\bf u})$ is a free module,
    \item ${\bf u}$ is a basis of~${\rm span}_{k\left[\frac{d}{dt}\right]}({\bf u})$.
\end{itemize}
This independence will be assumed in the rest of the paper. A \emph{linear input-output} system is a linear dynamics that is equipped with a finite set ${\bf y} = \{y_1, \dots, y_p\}$ of \emph{output variables}.

\subsection{Equations}\label{equat}

Associate the system $\Lambda$ with a short exact sequence
\begin{equation}\label{sequence}
 0 \rightarrow \mathcal{E} \rightarrow \frak{F} \rightarrow \Lambda \rightarrow 0
\end{equation}
where $\mathfrak{F}$ is free. It is related to the linear differential equations that define $\Lambda$:
\begin{equation}
\label{eq}
  \sum_{\kappa =1}^{\mu} a_{\iota\kappa} w_\kappa = 0, \quad a_{\iota \kappa} \in
  k \left[\frac{d}{dt}\right], \quad
 \iota  = 1,\ldots, \mu
\end{equation}
Let $\mathfrak{F}$ be the free module with basis $W_1, \dots, W_\mu $. Let $\mathcal{E} \subset \mathfrak{F}$ be spanned by $\sum_{\kappa =1}^{\mu} a_{\iota\kappa} W_\kappa$, $\iota = 1, \dots, \nu$.  Then $\Lambda = \mathfrak{F}/\mathcal{E}$. Note that Eq.~\eqref{eq} may be rewritten as
\begin{equation}\label{prese}
P \begin{pmatrix}
    w_1 \\ \vdots \\ w_{\mu}
\end{pmatrix}
= 0
\end{equation}
where $P \in k\left[\frac{d}{dt}\right]^{\nu \times \mu}$ is a \emph{presentation} matrix.

\subsection{State-variable representation}

Take a linear system $\Lambda$ with control and output variables ${\bf u}$ and ${\bf y}$. Set $n = \dim_k (\Lambda / {\rm span}_{k\left[\frac{d}{dt}\right]}({\bf u}))$. A \emph{generalized state} ${\bf \xi} = \{\xi_1, \dots, \xi_n \}$ is an $n$-tuple of elements in $\Lambda$ such that its residue in $\Lambda / {\rm span}_{k\left[\frac{d}{dt}\right]}({\bf u})$ is a basis of the latter as a $k$-vector space. It yields a \emph{generalized state-variable dynamics}
\begin{equation}\label{gene}
    \frac{d}{dt} \begin{pmatrix}
        \xi_1 \\ \vdots \\ \xi_n
    \end{pmatrix} 
    = \mathcal{J} \begin{pmatrix}
        \xi_1 \\ \vdots \\ \xi_n
    \end{pmatrix} + \sum_{\alpha = 0}^{\mu} \mathcal{G}_\alpha \frac{d^\alpha}{dt^\alpha}
    \begin{pmatrix}
        u_1 \\ \vdots \\ u_m
    \end{pmatrix}
\end{equation}
where $\mathcal{J} \in k^{n \times n}$, $\mathcal{G}_\alpha \in k^{n \times m}$. Two generalized states ${\bf \xi}$ and ${\bf \underline\xi} = \{\underline{\xi}_1, \dots, \underline{\xi}_n\}$ are related by a control dependent relation:
\begin{equation}\label{change}
\begin{pmatrix}
        \underline{\xi}_1 \\ \vdots \\ \underline{\xi}_n
    \end{pmatrix} = \mathcal{O} \begin{pmatrix}
        \xi_1 \\ \vdots \\ \xi_n
    \end{pmatrix} + \sum_{\rm finite} \mathcal{S}_\beta \frac{d^\beta}{dt^\beta} \begin{pmatrix}
        u_1 \\ \vdots \\ u_m
    \end{pmatrix}
\end{equation}
where $\mathcal{O} \in k^{n \times n}$ is invertible, $\mathcal{S}_\beta \in k^{n \times m}$. For the output variables, Eq.~\eqref{gene} should be completed by
\begin{equation}\label{output}
  \begin{pmatrix}
        y_1 \\ \vdots \\ y_p
    \end{pmatrix} =  \mathcal{H} 
    \begin{pmatrix}
        \xi_1 \\ \vdots \\ \xi_n
    \end{pmatrix} + \sum_{\rm finite} \mathcal{J}_\gamma \frac{d^\gamma}{dt^\gamma} \begin{pmatrix}
        u_1 \\ \vdots \\ u_m
    \end{pmatrix} 
\end{equation}
where $\mathcal{H} \in k^{p \times n}$, $\mathcal{J}_\gamma \in k^{p \times m}$. Assume that $\mu \geq 1$ and $\mathcal{G}_\mu \neq 0$ in Eq.~\eqref{gene}. Define a new generalized state via
$$
\begin{pmatrix}
        {\xi}_1 \\ \vdots \\ {\xi}_n
    \end{pmatrix}  = \begin{pmatrix}
        {\xi}_{1}^{\star} \\ \vdots \\ {\xi}_n^{\star}
    \end{pmatrix} + \mathcal{G}_\mu \frac{d^{\mu - 1}}{dt^{\mu - 1}}
    \begin{pmatrix}
        u_1 \\ \vdots \\ u_m
    \end{pmatrix} 
$$
The highest derivative of the control variables in the corresponding new generalized state variable representation is thus less or equal to $\mu - 1$. This procedure leads, by induction, to the famous \emph{Kalman state variable representation} which does not involve any derivative of the control variables:
\begin{equation}\label{kalman}
    \frac{d}{dt} \begin{pmatrix}
        x_1 \\ \vdots \\ x_n
    \end{pmatrix} 
    = F \begin{pmatrix}
        x_1 \\ \vdots \\ x_n
    \end{pmatrix} + G 
    \begin{pmatrix}
        u_1 \\ \vdots \\ u_m
    \end{pmatrix}
\end{equation}
where
\begin{itemize}
    \item $\{x_1, \dots, x_n\}$ is called a \emph{Kalman state};
    \item $F \in k^{n \times n}$, $G \in k^{n \times m}$ are constant matrices of appropriate sizes. 
\end{itemize}
Eq.~\eqref{change} shows that two Kalman states are related by a classic state-variable transformation (see, e.g., \cite{kailath,lunze,sontag,wolovich}):
\begin{equation}\label{kalmanchange}
  \begin{pmatrix}
        \underline{x}_1 \\ \vdots \\ \underline{x}_n
    \end{pmatrix} 
    = T \begin{pmatrix}
        x_1 \\ \vdots \\ x_n
    \end{pmatrix} 
\end{equation}
where $T \in {\rm GL}_n (k)$ is an invertible square matrix. 

The output map is deduced from Eq.~\eqref{output}:
\begin{equation}\label{kalmanoutput}
   \begin{pmatrix}
        y_1 \\ \vdots \\ y_p
    \end{pmatrix} =  H 
    \begin{pmatrix}
        x_1 \\ \vdots \\ x_n
    \end{pmatrix} + \sum_{\rm finite} J_\gamma \frac{d^\gamma}{dt^\gamma} \begin{pmatrix}
        u_1 \\ \vdots \\ u_m
    \end{pmatrix} 
\end{equation}
where $H \in k^{p \times n}$, $J_\gamma \in k^{p \times m}$.

The following statement summarizes the previous calculations:
\begin{theorem}
Let $\Lambda$ be an input-output system where ${\bf u} = \{u_1, \dots, u_m\}$ and ${\bf y} = \{y_1, \dots, y_p\}$ are, respectively, the control and output variables. There exists a Kalman representation~\eqref{kalman} for the dynamics where the dimension of the Kalman state $\{x_1, \dots, x_n\}$ is equal to $\dim_k \left(\Lambda / {\rm span}_{k\left[\frac{d}{dt}\right]}({\bf u})\right)$. Two Kalman states are related by Eq.~\eqref{kalmanchange}. The output is given by Eq.~\eqref{kalmanoutput}.
\end{theorem}
The $n$-dimensional $k$-vector space $\frak{K}$ spanned by a Kalman state $\{x_1, \dots, x_n\}$ is called the \emph{Kalman vector space}.
\begin{corollary}
 In the category of non-necessarily finitely generated $k$-vector spaces, the dynamics $\Lambda$ is the direct sum of the Kalman vector space $\mathfrak{K}$ and ${\rm span}_{k\left[\frac{d}{dt}\right]} ({\bf u})$, i.e., 
$\Lambda = \mathfrak{K} \bigoplus {\rm span}_{k\left[\frac{d}{dt}\right]} ({\bf u})$.
\end{corollary}

\subsection{Controllability and observability}


\subsubsection{Controllability}

The Kalman matrix controllability criterion states that Eq.~\eqref{kalman} defines a controllable dynamics if and only if
$
\mathrm{rk}
\begin{pmatrix}
    G, FG, \dots, F^{n - 1}G
\end{pmatrix}
= n
$,
where $n$ is the dimension of the Kalman state.
It is well known that non-controllability yields the existence of state variables $\{\eta_1, \dots, \eta_\nu\}$ such that
\begin{equation}\label{uncontr}
    \frac{d}{dt} \begin{pmatrix}
        \eta_1 \\ \vdots \\ \eta_\nu
    \end{pmatrix}
    = \mathfrak{N}
    \begin{pmatrix}
        \eta_1 \\ \vdots \\ \eta_\nu
    \end{pmatrix}
\end{equation}
where $\mathfrak{N} \in k^{\nu \times \nu}$, $1 \leqslant \nu \leqslant n$. This is equivalent to the existence of a nontrivial torsion sub-module of $\Lambda$. Therefore, controllability is equivalent to the freeness of the module~$\Lambda$. We are therefore led to the following definition, which is independent of any peculiar representation:
\begin{definition}
    The system $\Lambda$ is said to be \emph{controllable} if and only if the module $\Lambda$ is free. Any basis of~$\Lambda$ is called a \emph{flat output}.
\end{definition}

\subsubsection{Observability}

Eq.~\eqref{kalmanoutput} yields the following result for $\kappa \geqslant 1$
\begin{equation*}\label{kalmanobs}
  \frac{d^\kappa}{dt^\kappa}  \begin{pmatrix}
        y_1 \\ \vdots \\ y_p
    \end{pmatrix} =  H F^{\kappa - 1}
    \begin{pmatrix}
        x_1 \\ \vdots \\ x_n
    \end{pmatrix} + \mathcal{U}_k {\begin{pmatrix} u_1 \\ \vdots \\ u_m\end{pmatrix}}
\end{equation*}
where $\mathcal{U}_\kappa \in k\left[\frac{d}{dt}\right]^{p \times m}$. We are thus led to the Kalman observability matrix: 
\begin{equation*}
    \mathfrak{O} = \begin{pmatrix}
        H \\ HF \\ \vdots \\ HF^{n - 1}
    \end{pmatrix}
\end{equation*}
The Kalman observability is characterized by ${\rm rk}(\mathfrak{O}) = n$. It is equivalent to saying that any Kalman state belongs to ${\rm span}_{k\left[\frac{d}{dt}\right]}({\bf u}, {\bf y})$, i.e., is a $k$-linear combination of the control and output variables and their derivatives up to some finite order. We are therefore led to the following more intrinsic definition:
\begin{definition}
    The system $\Lambda$ with control and output variables ${\bf u}$ and ${\bf y}$ is said to be observable if and only if $\Lambda = {\rm span}_{k\left[\frac{d}{dt}\right]}({\bf u}, {\bf y})$.
\end{definition}
\noindent In other words, System $\Lambda$ is observable if and only if any system variable, i.e., any element of $\Lambda$, is a $k$-linear combination of the control and output variables and of their derivatives up to some finite order.
The next result is obvious:
\begin{proposition}
If ${\bf y}$ is a flat output, the input-output system is observable.
\end{proposition}

\subsection{Well-formed dynamics}\label{wellformed}


\subsubsection{Presentation}
Assume that $\Lambda$ is a controllable dynamics system with control variables ${\bf u} = \{u_1, \dots, u_m\}$. It is said to be \emph{well-formed}\footnote{This terminology is borrowed from \cite{rudolph}, where it was introduced in the context of nonlinear control theory.} if and only if the rank of the matrix~$G$ in the Kalman state-variable representation~\eqref{kalman} is equal to~$m$.


Set $\Lambda_\rho = {\rm span}_{k\left[\frac{d}{dt}\right]} \{\lambda^{(\sigma)} \mid \lambda \in \Lambda, \sigma \geqslant \rho\}$. The decreasing sequence

\begin{equation}\label{filtra}
    \Lambda = \Lambda_0 \supset \Lambda_1 \supset \cdots \supset \Lambda_\rho \supset \cdots
\end{equation}
defines a {\em filtration} of $\Lambda$.

\begin{proposition}
\label{well} The four following properties are equivalent:
\begin{enumerate}
\item The dynamics~$\Lambda$ is well-formed, 
 \item $\Lambda = {\rm span}_{k\left[\frac{d}{dt}\right]}(x_1, \dots, x_n)$, 
\item ${\rm span}_k ({\bf u}) \cap (\Lambda \setminus \Lambda_1) = \{0\}$,
\item $(\Lambda \setminus \Lambda_1) \subseteq \mathfrak{K}$.
\end{enumerate}
\end{proposition}
\begin{proof}
The rank condition on $G$ in Eq. \eqref{kalman} implies 
$$u_1, \dots, u_m \in {\rm span}_k (x_1, \dots, x_n, \dot{x}_1, \dots, \dot{x}_n)$$ 
and 
$$u_1, \dots, u_m \notin {\rm span}_k (x_1, \dots, x_n)$$
The above properties follow at once. 
\end{proof}

From now on, $\Lambda$ is assumed to be well-formed. Then 
\begin{corollary}
$m \leqslant n$.
\end{corollary}

Via the filtration \eqref{filtra}, $\Lambda$ may be associated with a {\em graded} module $$\bigoplus_{\rho \geqslant 0} \Lambda_\rho / \Lambda_{\rho + 1}$$ that is \emph{graded-free} of rank $m$. Any element in $\Lambda$ corresponding to $\Lambda_\rho / \Lambda_{\rho + 1}$ is said to be \emph{homogeneous of order $\rho$}. The set of such elements is a $k$-vector space of dimension $m$.

The next result follows immediately:
\begin{lemma}
    $\Lambda \setminus \Lambda_1$ is a $k$-vector space of dimension $m$.
\end{lemma}
$\Lambda \setminus \Lambda_1$ is called the \emph{roof} of $\Lambda$. We have proven that the flat output may be chosen in the Kalman vector space:
\begin{proposition}
    Any basis of the roof, as a vector space, is a flat output.
\end{proposition}

\subsubsection{An elementary example: 
Controllable canonical form}\label{elem}\label{cano}
Let $\Lambda$ be a well-formed controllable dynamics with a single control $u$. Any basis $x_1$ of $\Lambda$ is unique up to multiplication by a non-zero scalar, i.e., by an element in $k \setminus \{0\}$. Set $n = \dim_k \left(\Lambda / {\rm span}_{k\left[\frac{d}{dt}\right]} (u)\right)$. It yields
the well-known {\em controllable canonical form} (see,
e.g, \cite{andrea,kailath,lunze,sontag,trelat05,wolovich})
\begin{equation}\label{ex1}
\begin{cases}
\dot{x}_1 &= x_2 \\
&\vdots \\
\dot{x}_n &= -a_0 x_1 - \cdots - a_{n-1}x_n + b u  
\end{cases}
\end{equation}
where $a_0, \dots, a_{n-1}, b \in k$, $b \neq 0$. 

\begin{remark}
The most celebrated flat output corresponds, of course, to the Brunovsk\'{y} canonical form \cite{brunov} (see also \cite{sontag}), which should be viewed as a multivariable extension of the controllable canonical form. See \cite{ng} for further considerations on flat outputs that might be useful in our optimal control setting.
\end{remark}

\subsection{Laplace functor}
\subsubsection{Definition}
Let $k(s)$, $s = \frac{d}{dt}$, be the quotient field of the integral ring $k\left[\frac{d}{dt}\right]$. Consider the abelian category {\bf Vect} of finite-dimensional $k(s)$-vector spaces. Let $\Lambda$ be a system. Introduce the localization or tensor product $$\hat{\Lambda} = k(s) \bigotimes_{k\left[\frac{d}{dt}\right]} \Lambda$$ $\hat{\Lambda}$ is a $k(s)$-vector space.\footnote{To simplify notations, write $\bigotimes$ instead of $\bigotimes_{k\left[\frac{d}{dt}\right]}$ in the sequel.} The kernel of the $k\left[\frac{d}{dt}\right]$-linear morphism $\Lambda \longrightarrow \hat{\Lambda}$, $\lambda \longmapsto \hat{\lambda} = 
1 \bigotimes \lambda$ is the torsion submodule of $\Lambda$. This is a right-exact additive functor between the abelian categories {\bf Mod} and {\bf Vect},  which is called the \emph{Laplace functor}.\footnote{The Laplace functor replaces the classic Laplace transform, which is a mainstay in control for introducing transfer functions and matrices (see, e.g., the textbooks \cite{andrea,brockett,kailath,lunze,sontag,wolovich}).} Then
$$
\dim_{k(s)} (\hat{\Lambda}) = {\rm rk} (\Lambda)
$$
If $\Lambda_2 \subset \Lambda_1$ and if the quotient module $\Lambda_1 / \Lambda_2$ is torsion, then $\hat{\Lambda}_2 = \hat{\Lambda}_1$. 
\subsubsection{Transfer matrix}
Consider now a system $\Lambda$, with control and output variables ${\bf u} = \{u_1, \dots, u_m\}$ and ${\bf y} = \{y_1, \dots, y_p\}$ respectively. The previous properties show that the independence of the control variables is equivalent to the fact that $\hat{u}_1, \dots, \hat{u}_m$ is a basis of the $k(s)$-vector space $\hat{\Lambda}$. In this case, then
\begin{equation*}
    \begin{pmatrix}
        \hat{y}_1 \\ \vdots \\ \hat{y}_p
    \end{pmatrix} = \mathfrak{T} \begin{pmatrix}
        \hat{u}_1 \\ \vdots \\ \hat{u}_m
    \end{pmatrix}
\end{equation*}
where $\mathfrak{T} \in k(s)^{p \times m}$ is the {\em transfer matrix}. If $m = p = 1$, $\mathfrak{T} = \frac{\mathfrak{p}}{\mathfrak{q}}$, where $\mathfrak{p}, \mathfrak{q} \in k[s]$, $(\mathfrak{p}, \mathfrak{q}) = 1$, is called the {\em transfer function}. The following properties are obvious:
\begin{proposition}
Assume that $\Lambda$ is controllable. Then $\bf y$ is a flat output if and only if 
\begin{itemize}
    \item the transfer matrix $\mathfrak{T}$ is square and invertible,
    \item the entries of its inverse belong to $k[s]$.
\end{itemize}
\end{proposition}
\begin{corollary}
If $m = p = 1$, the output of the controllable monovariable system is flat if and only if the numerator $\mathfrak{p}$ of the transfer function $\frac{\mathfrak{p}}{\mathfrak{q}}$ is a non-zero constant, i.e., if it belongs to $k \setminus \{0\}$. 
\end{corollary}
\subsubsection{Rank of the presentation matrix}\label{pres}
With the notations of Sec. \ref{equat}, Eqs.~\eqref{sequence} and~\eqref{prese} and the right-exactness of the Laplace functor yield at once

\begin{proposition}
${\rm rk} (\Lambda) = \mu - {\rm rk} (P)$.
\end{proposition}
and
\begin{corollary}\label{corol}
    If $P$ is square, i.e., $\mu = \nu$, then $\Lambda$ is torsion if and only if $\det (P) \neq 0$.
\end{corollary}

\subsection{Lagrangians}\label{lagr}
Call a \emph{Lagrangian}, or \emph{cost function}, any element $\mathcal{L}$ of the symmetric algebra $S(\Lambda)$ of $\Lambda$ over $k$. Note that $S(\Lambda)$ 
\begin{itemize}
    \item is a {\em differential ring} where $k$ is the field of {\em constants}, i.e., $\forall a \in k$, $\dot{a} = \frac{da}{dt} = 0$; 
    \item is a graded module:
 $$S(\Lambda) = \bigoplus_{\nu \geqslant 0} S^\nu (\Lambda) $$
where $S^\nu (\Lambda)$ is the symmetric power of order $\nu$.
\end{itemize}
Moreover,
\begin{itemize}
\item The \emph{degree} of $\mathcal{L}$ is the minimal order $d$ such that $\mathcal{L} \in \bigoplus_{\nu \leqslant d} S^\nu (\Lambda) $. 
    \item $\mathcal{L}$ is said to be \emph{quadratic} if and only if $\mathcal{L} \in S^2(V)$.
\end{itemize}

Let $\Lambda$ be controllable and ${\bf y} = \{y_1, \dots, y_m\}$ be a basis. Thus $\{y_{\iota}^{(\nu_\iota)} \mid \iota = 1, \dots, m; \nu_\iota \geqslant 0 \}$ is a basis of the $k$-vector  space $\Lambda$, and $S(\Lambda)$ is isomorphic to the differential ring $k\langle {\bf y} \rangle$, i.e., to the ring of $k$-polynomials $k[y_{\iota}^{(\nu_\iota)} \mid \iota = 1, \dots, m; \nu_\iota \geqslant 0 ]$ in an infinite number of variables.

The next trivial consequences of the definitions connect the usual LQR presentation with our formalism:
\begin{itemize}
    \item If $\mathcal{L}$ is a $k$-polynomial of the control and state variables, then $\mathcal{L} \in S(\Lambda)$.
    \item If $\mathcal{L}$ is a $k$-polynomial of degree $2$ of the control and state variables, then it is of degree $2$ in $S(\Lambda)$.
    \item If $\mathcal{L}$ is a homogeneous $k$-polynomial of degree $2$ of the control and state variables, then it is quadratic.
\end{itemize}

\subsection{Variational system}\label{variat}

Consider the quotient of the differential rings  
$$
\Lambda_\Delta = \bigoplus_{\mu \geqslant 1} S^\mu (\Lambda) / \bigoplus_{\nu \geqslant 2} S^\nu (\Lambda)
$$
The restriction of the canonical epimorphism 
$\bigoplus_{\mu \geqslant 1} S^\mu (\Lambda) \to \Lambda_\Delta$
to $S^1(\Lambda) = \Lambda$ defines a $k\left[\frac{d}{dt}\right]$-linear isomorphim $\Delta: \Lambda \to \Lambda_\Delta$ between modules, where
\begin{itemize}
    \item $\Lambda_\Delta$ is the \emph{variational system} associated to $\Lambda$, 
    \item $\Delta$ is the (\emph{K\"{a}hler}) \emph{differential}.
\end{itemize}
$\Lambda_\Delta$ is thus controllable if and only if $\Lambda$ is controllable.
  
\section{Open-loop optimality on a finite time horizon}
\label{euler}
\subsection{Trajectory}
From now on, the base field~$k$ is the field $\mathbb{R}$ of real numbers. 
The set $C^\infty (\mathcal{O})$ of smooth functions $\mathcal{O} \to \mathbb{R}$, where $\mathcal{O}$ is an open subset of $\mathbb{R}$, is a $\mathbb{R}\left[\frac{d}{dt}\right]$-module that is not finitely generated. According to~\cite{fliess92}, a \emph{trajectory}  of a system $\Lambda$ is nothing but a $\mathbb{R}\left[\frac{d}{dt}\right]$-module morphism $\phi: \Lambda \to C^\infty(\mathcal{O})$.
\begin{remark}
The choice of the functional analytic nature
of the trajectories is, of course, manifold.
\end{remark}
\noindent In the sequel, $\phi (\lambda)$ will no longer be used and will be replaced, with a slight abuse of notations, by~$\lambda$ or even $\lambda (t)$. 
\subsection{Euler-Lagrange equations and modules}\label{EL}
Let $\Lambda$ be a controllable system, and ${\bf y} =\{y_1, \dots, y_m\}$ a flat output. Introduce the \emph{Lagrangian} (see Sect. \ref{lagr})
$
\mathcal{L}(y_1, \dot{y}_1, \dots, y_{1}^{(\mu_1)}, \dots, y_m, \dot{y}_m, \dots, y_{m}^{(\mu_m)}) 
$.
Corresponding to the integral criterion
$J = \int_{0}^{T} \mathcal{L} dt$,
where $T > 0$ is the \emph{time horizon},
the \emph{stationary} solutions satisfy the Euler-Lagrange system of ordinary differential equations.
\begin{equation}
\label{ele}
\frac{\partial \mathcal{L}}{\partial y_\iota} - \frac{d}{dt} \frac{\partial \mathcal{L}}{\partial \dot{y}_\iota} + \dots + (- 1)^{\mu_\iota} \frac{d^{\mu_\iota}}{dt^{\mu_\iota}} \frac{\partial \mathcal{L}}{\partial y_{\iota}^{(\mu_\iota)}} = 0, \hspace{0.16cm} \iota = 1, \dots, m
\end{equation}   
The next proposition, although trivial, is pivotal:
\begin{proposition}
    If $\mathcal{L}$ is quadratic (resp. of degree $2$), Eq.~\eqref{ele} is a system of homogeneous (resp. non-necessarily homogeneous) linear ordinary differential equations with constant coefficients in $\mathbb{R}$.
\end{proposition}

\noindent In the quadratic case Eq.~\eqref{ele} may be written in matrix form
\begin{equation*}\label{matrix}
\mathcal{M} \begin{pmatrix}
y_1 \\
\vdots \\
y_m \\
\end{pmatrix} = 0  
\end{equation*}
where ${\mathcal{M}} \in \mathbb{R}\left[\frac{d}{dt}\right]^{m \times m}$. 
It defines, according to Corollary \ref{corol}, the module $\mathcal{E} = {\rm span}_{\mathbb{R}\left[\frac{d}{dt}\right]}({\bf y})$, which is called the \emph{Euler-Lagrange module}. Corollary \ref{corol} states that it is torsion if and only if $\det (\mathcal{M}) \neq 0$. Then $\mathcal{E}$ is also called the 
\emph{Euler-Lagrange vector space}, which is finite-dimensional.

\subsection{The two-point boundary problem: the monovariable case}

We focus on a single linear differential equation with constant coefficients with a two-point boundary condition.

\subsubsection{General situation}

To an Euler-Lagrange equation $a_0 y + a_1 \dot{y} + \dots + a_N y^{(N)} = 0$, $a_0, a_1, \dots, a_N \in \mathbb{R}$, $a_N \neq 0$, corresponds the Euler-Lagrange vector space $\mathcal{E}$ of dimension $N$. Any element of a basis $B = \{\sigma_1(t), \dots, \sigma_N (t)\}$ is an exponential polynomial of the form $\Sigma_{\rm finite} P(t)e^{at}\sin(\omega t + \varphi)$, $P(t) \in \mathbb{R}[t]$, $a, \omega, \varphi \in \mathbb{R}$: it is an entire analytic function of the complex variable $t \in \mathbb{C}$. It might therefore be natural to replace $\mathbb{R}$ with $\mathbb{C}$: $B$ is also a basis of the $\mathbb{C}$-vector space $\mathcal{E}_\mathbb{C} = \mathbb{C}\bigotimes_\mathbb{R} \mathcal{E}$. 
Set $y(t) = c_1 \sigma_1 (t) + \dots + c_N \sigma_N (t)$, $c_1, \dots, c_N \in \mathbb{C}$, such that $y^{(\mu_\iota)} (0)$, $\mu_\iota \leqslant M < N$, $y^{(\nu_\kappa)} (T)$, $\nu_\kappa \leqslant N - M - 1 $ are given. It yields
\begin{equation}\label{sin}
\begin{pmatrix}
        y (0) \\
        \vdots \\
        y^{(M)}(0) \\
        y (T) \\
        \vdots \\
        y^{(N - M -1)} (T)
    \end{pmatrix} =
W(T)
    \begin{pmatrix}
        c_0 \\
        \vdots \\
        c_{N}
    \end{pmatrix}
\end{equation}
where $W(T)$ is a $N\times N$ matrix:
\begin{equation}\label{W}
   W(T) = \begin{pmatrix}
        \sigma_1(0) & \cdots  & \sigma_{N}(0) \\
        \dot{\sigma}_1(0) & \cdots & \dot{\sigma}_{N}(0) \\
         \vdots \\
         \sigma_{1}^{(M)}(0) & \cdots & \sigma_{N}^{(M)}(0) \\
          \sigma_1(T) & \cdots  & \sigma_{N}(T) \\
        \dot{\sigma}_1(T) & \cdots & \dot{\sigma}_{N}(T) \\
         \vdots \\
         \sigma_{1}^{(N - M - 1)}(T) & \cdots & \sigma_{N}^{(N - M - 1)}(T) 
    \end{pmatrix} 
\end{equation}

\begin{proposition}\label{prop}
  The solution of the two-point boundary value problem corresponding to Eq.~\eqref{sin} is \emph{generically solvable}: it exists and is unique for all $T > 0$, except perhaps for isolated values of $T$.
\end{proposition}

\begin{proof}
Introduce the $N \times N$ matrix:
\begin{equation}\label{2}
   \mathcal{W}(T) = \begin{pmatrix}
        \sigma_1(0) & \cdots  & \sigma_{N}(0) \\
        \dot{\sigma}_1(0) & \cdots & \dot{\sigma}_{N}(0) \\
         \vdots \\
         \sigma_{1}^{(M)}(0) & \cdots & \sigma_{N}^{(M)}(0) \\
          \sigma_{1}^{(M + 1)}(T) & \cdots  & \sigma_{N}^{(M + 1)}(T) \\
        \sigma_1^{(M + 2)}(T) & \cdots & \sigma_{N}^{(M + 2)}(T) \\
         \vdots \\
         \sigma_{1}^{(N)}(T) & \cdots & \sigma_{N}^{(N)}(T) 
    \end{pmatrix} 
\end{equation}
Assume, without any loss of generality, that the Wronskian (see, e.g., \cite{bostan,kaplansky}) satisfies $\det (\mathcal{W}(0)) \neq 0$. Then, $\det (\mathcal{W} (T))$, which is an entire function of the complex variable $T$, is equal to $0$ only at isolated complex values of $T$. If $\det (W(T)) \equiv 0$, the lines of $W(T)$ are linearly dependent over the {\em differential field} $\mathfrak{C}$ of meromorphic functions of the complex variable $T \in \mathbb{C}$.\footnote{$\mathfrak{C}$ is the quotient field of the differential ring of entire functions (see, e.g., \cite{rudin}).} Straightforward calculations on the $\mathfrak{C}$-linear combinations of the lines of $W(T)$ after successive derivations with respect to $T$ show that the lines of $\mathcal{W}(T)$ are linearly dependent. It contradicts $\det (\mathcal{W} (0)) \neq 0$.
\end{proof}

\subsubsection{Back to the controller canonical form}
Associate with Eq.~\eqref{ex1} the criterion
\begin{equation}\label{J1}
    J_1 = \int_{0}^{T} \mathfrak{Q}(x_1, \dots, x_n) + {\mathfrak{r}} u^2,  \quad \mathfrak{r} \in \mathbb{R}
\end{equation}
where $\mathfrak Q$ is a quadratic form with respect to the state variables. Formula~\eqref{ele} shows that the order with respect to $y = x_1$ of the corresponding linear homogeneous Euler-Lagrange equation is $2n$. Proposition \ref{prop} yields  
\begin{corollary}
    The two-point boundary problem associated with Eq.~\eqref{J1} for $x_1(0), \dots, x_n(0), x_1(T), \dots, x_n(T)$ is generically solvable.
\end{corollary}
\begin{remark}\label{rest}
    By taking a Lagrangian with a larger derivation order of $y$, it becomes possible to add $u(0)$ and $u(T)$ to the two-point boundary value problem (see Sect. \ref{subrest}).
\end{remark}

\subsection{The two-point boundary value problem: The multivariable case}

Assume that the Euler-Lagrange module $\mathcal{E}$ is torsion. 
For a single $y_\iota$, the two-point boundary value problem is assumed to be almost always solvable according to Prop.~\ref{prop}. Initial (resp. final) conditions apply to a finite set $\mathcal{I} = \{y_{\iota}^{(\nu_\iota)}\}$ (resp. $\mathcal{J} = \{y_{\varepsilon}^{(\mu_\varepsilon)}\}$). Set $n_i = \dim ({\rm span}_\mathbb{C}(\mathcal{I})) \leqslant {\rm card} (\mathcal{I})$ (resp. $n_j = \dim ({\rm span}_\mathbb{C}(\mathcal{J})) \leqslant {\rm card} (\mathcal{J})$). If $n_i < {\rm card} (\mathcal{I})$ (resp. $n_j < {\rm card} (\mathcal{J}))$, then $y_{\iota}^{(\nu_\iota)}(0)$ (resp. $y_{\varepsilon}^{(\mu_\varepsilon)} (T))$ satisfies some $\mathbb{C}$-linear relations and cannot be chosen arbitrarily. It leads to the following formal definition: The initial (resp. final) conditions are said to be \emph{compatible} if and only if there exists $\lambda_i^\star \in ({\rm span}_\mathbb{C}(\mathcal{I}))^\star$ (resp. $\lambda_f^\star \in ({\rm span}_\mathbb{C}(\mathcal{J}))^\star$, where (${\rm span}_\mathbb{C}(\mathcal{I}))^\star$ (resp. (${\rm span}_\mathbb{C}(\mathcal{J}))^\star$) is the dual vector space of ${\rm span}_\mathbb{C}(\mathcal{I})$ (resp. ${\rm span}_\mathbb{C}(\mathcal{J})$), such that $\lambda_i^\star (y_{\iota}^{(\nu_\iota)}) = y_{\iota}^{(\nu_\iota)} (0)$ (resp. $\lambda_f^\star (y_{\varepsilon}^{(\mu_\varepsilon)}) = y_{\varepsilon}^{(\mu_\varepsilon)} (T)$).
\begin{proposition}
    If the Euler-Lagrange module is torsion, the two-point boundary value problem for Eq.~\eqref{ele} is {\em generically solvable} for compatible initial and final conditions.
\end{proposition}
\begin{remark}
    If $\mathcal{L} = \sum_{\iota = 1}^{m} \mathcal{L}_\iota$, where $\mathcal{L}_\iota$ is quadratic with respect to $y_\iota$ and its derivatives, then Eq.~\eqref{ele} shows that the initial and final conditions are always compatible.
\end{remark}

\subsection{Four illustrative examples}\label{4}
\subsubsection{Optimal time horizon}\label{oth}
Assume that the horizon $T$ belongs to an open interval $\mathcal{O} \subset \mathbb{R}$, such that the two-point boundary value problem is solvable. Then $J$ may be viewed as a differentiable function of $T$. The horizon $T$ is said to be \emph{stationary} at $T_0 \in \mathcal{O}$ if and only if $\frac{dJ}{dT}(T_0) = 0$. If $J(T)$ is reaching an extremum at $T_0$, then $T_0$ is said to be an \emph{optimal} time horizon.
As a numerical illustration, let's introduce the control canonical form (Sect. \ref{cano})  
\begin{equation*}\label{ex2}
\begin{cases}
\dot{x}_1 &= x_2 \\
\dot{x}_2 &= - 2 x_1 - 4 x_2 + 3 u  
\end{cases}
\end{equation*}
where $x_1$ is a flat output, and the criterion
$$
  J_1 (T) = \int_0^T \left( u^2(\tau) + (100 - x_1(\tau))^2 + 0.1 \dot{x}_{1}^2(\tau) \right) d \tau
$$
where $x_1(0) = \dot{x}_1 (0) = 0$, $x_1(T) = 100$, $\dot{x}_1 (T) = 0$. Fig. \ref{horizon} shows that $T_o \approx \SI{3}{\second}$ is an optimal time horizon.
\begin{figure*}[!ht]
\centering
{\epsfig{figure=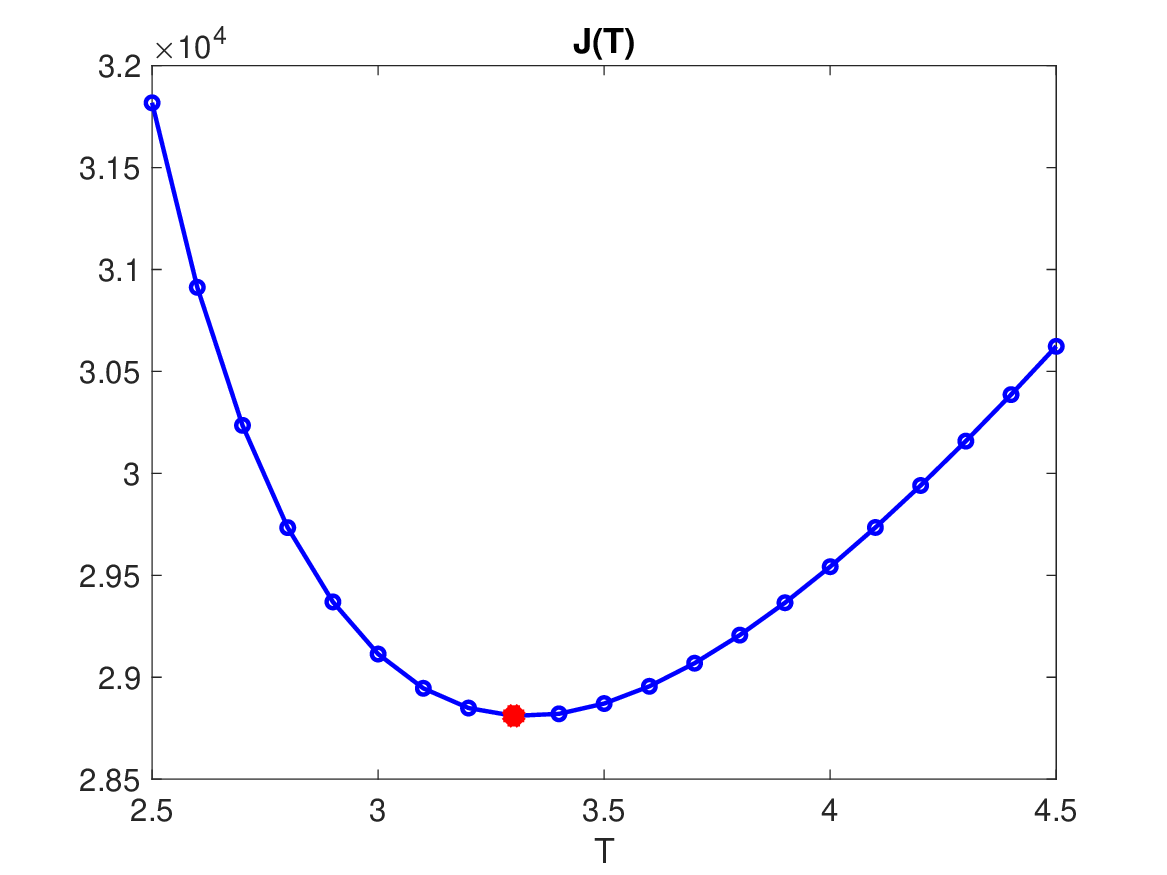,width=\tailleFFF\textwidth}}
\caption{The criterion $J_1$ as a function of the time horizon $T$}\label{horizon}
\end{figure*}

\subsubsection{Optimal parameter design}
Consider a criterion $J(T)$. Assume that 
\begin{itemize}
    \item the two-point boundary value problem is solvable,
    \item $J(T)$ is a differentiable function of $d \geqslant 1$ parameters belonging to an open subset of $\mathbb{R}^d$.
\end{itemize}
Those parameters are said to be \emph{stationary} if and only if the derivatives of $J(T)$ with respect to any of those parameters are $0$. There is an \emph{optimal parameter design} if and only if there is an extremum. 

Consider again a controller canonical form, where $\xi_1$ is a flat output, $v$ the control variable, and $b$ a parameter:
\begin{equation}\label{ex3}
\begin{cases}
\dot{\xi}_1 &= \xi_2 \\
\dot{\xi}_2 &= - (10 - b) \xi_1 - b\xi_2 + v  
\end{cases}
\end{equation}
Introduce the criterion 
$$
  J_2 (T) = \int_0^T \left( v^2(t) + 100(100 - \xi_1 (t))^2 \right) dt
$$
The boundary conditions are the same as in Sect. \ref{oth}: $\xi_1(0) = \dot{\xi}_1 (0) = 0$, $\xi_1(T) = 100$, $\dot{\xi}_1 (T) = 0$. Fig. \ref{bo} shows that $b \approx 6$ is an optimal parameter design.
\begin{figure*}[!ht]
\centering

{\epsfig{figure=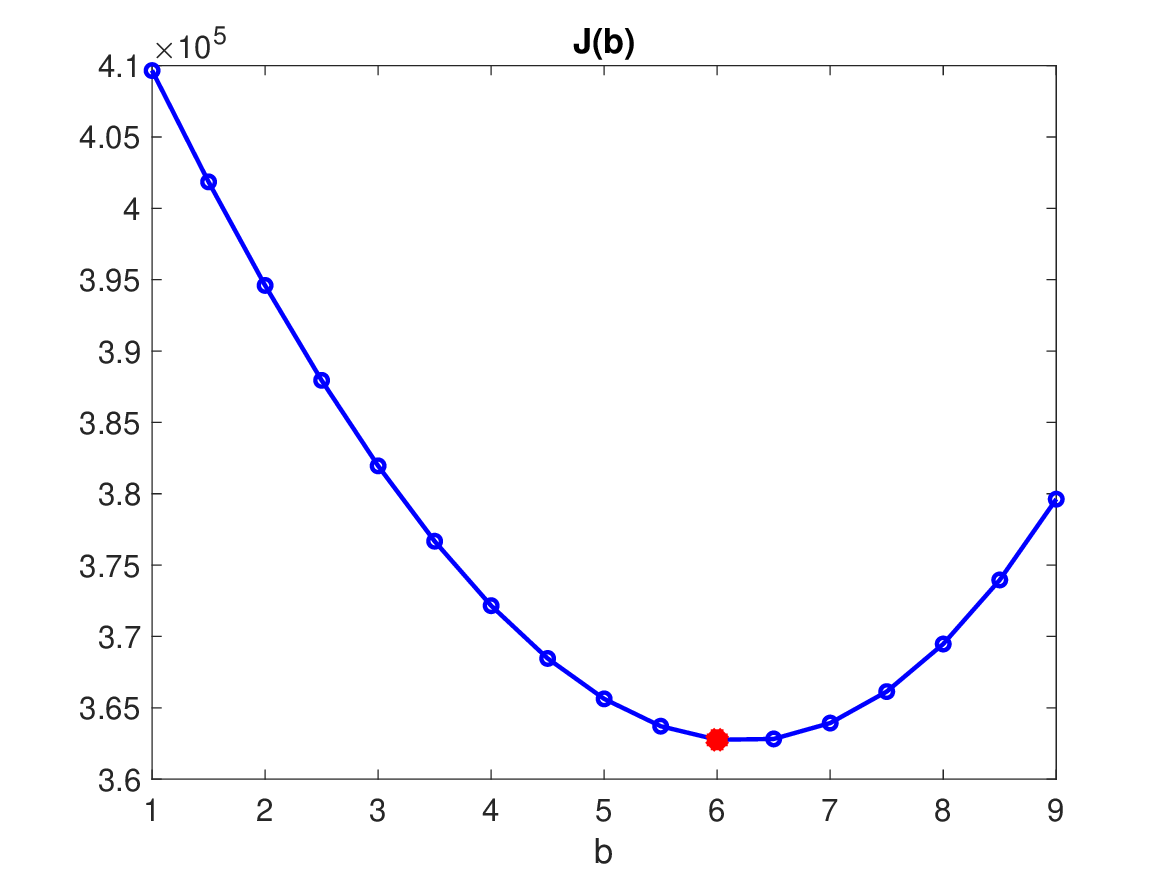,width=\tailleFFF\textwidth}}
\caption{The parametic dependence of criterion $J_2(T)$}\label{bo}
\end{figure*}

\subsubsection{Optimal rest-to-rest trajectory}\label{subrest}
Consider \cite{boscain} the simple integrator $\dot{y} = u$ and the criterion $$J_1 = \int_{0}^{T} (u(\tau))^2d\tau = \int_{0}^{T} (\dot{y}(\tau))^2d\tau$$ The corresponding Euler-Lagrange equation is $\ddot{y} = 0$. Its general solution reads $y_{\rm opt}(t) = C_0 + C_1 t$, $C_0, C_1 \in \mathbb{R}$.
It yields a constant optimal control $u_{\rm opt}(t) = C_1$.  With $y(0) = 0$, $y(T) = 2$ for instance, $C_0 = 0$, $C_1 = \frac{2}{T} = u_{\rm opt} (t)$.

Introduce the new criterion 
\begin{equation*}\label{quad}
J_2 = \int_{0}^{T} \left( (y(\tau))^2 + (u(\tau))^2 \right)d\tau = \int_{0}^{T} \left( (y(\tau))^2 + (\dot{y}(\tau))^2  \right)d\tau
\end{equation*}
The Euler-Lagrange equation becomes $ y - y^{(2)} = 0$. Its general solution is $ y_{\rm opt} (t) = C_1e^{-t}+C_2e^{t}$, $C_1, C_2 \in \mathbb{R}$. With $y(0) = 0$, $y(T) = 2$, $C_1 =\frac{2}{e^{-T} - e^{T}}$, $C_2 =\frac{-2}{e^{-T} - e^{T}}$.

To start and stop at rest, i.e., $\dot{y}(0) = u(0) = \dot{y}(T) = u(T) = 0$, according to Remark \ref{rest}, the previous criterion has to be modified:
\begin{equation}\label{critrest}
    J_3 = \int_{0}^{T} \left( (y(\tau))^2 + (u(\tau))^2 + (\dot{u}(\tau))^2 \right)d\tau = \int_{0}^{T} \left( (y(\tau))^2 + (\dot{y}(\tau))^2 + (\ddot{y}(\tau))^2 \right)d\tau
\end{equation}
It yields the Euler-Lagrange equation
$ y - y^{(2)} + y^{(4)} = 0$.
Its general solution reads
$$ y_{\rm opt} (t) = e^{-\frac{t\sqrt3}{2}} \left( C_1\cos \left(\frac{t}{2}\right)-C_2\sin \left(\frac{t}{2}\right)\right) + e^{\frac{t\sqrt3}{2}}\left( C_3\cos \left(\frac{t}{2}\right)-C_4\sin \left(\frac{t}{2}\right)\right)$$
where $C_1, C_2, C_3, C_4 \in \mathbb{R}$. When starting at rest $y(0) = \dot{y}(0) = u(0) = 0$, and stopping at rest $y(T) = 2$, $\dot{y}(T) = u(T) = 0$, Fig. \ref{J} displays the evolution of the criterion given by Eq. \eqref{critrest} with respect to $T$. Let us emphasize that it is decreasing rather abruptly. See Fig. \ref{OI} for $y$ and $u$ when $T = 0.5$ and $T = 5$.
\begin{figure*}[!ht]
\centering
{\epsfig{figure=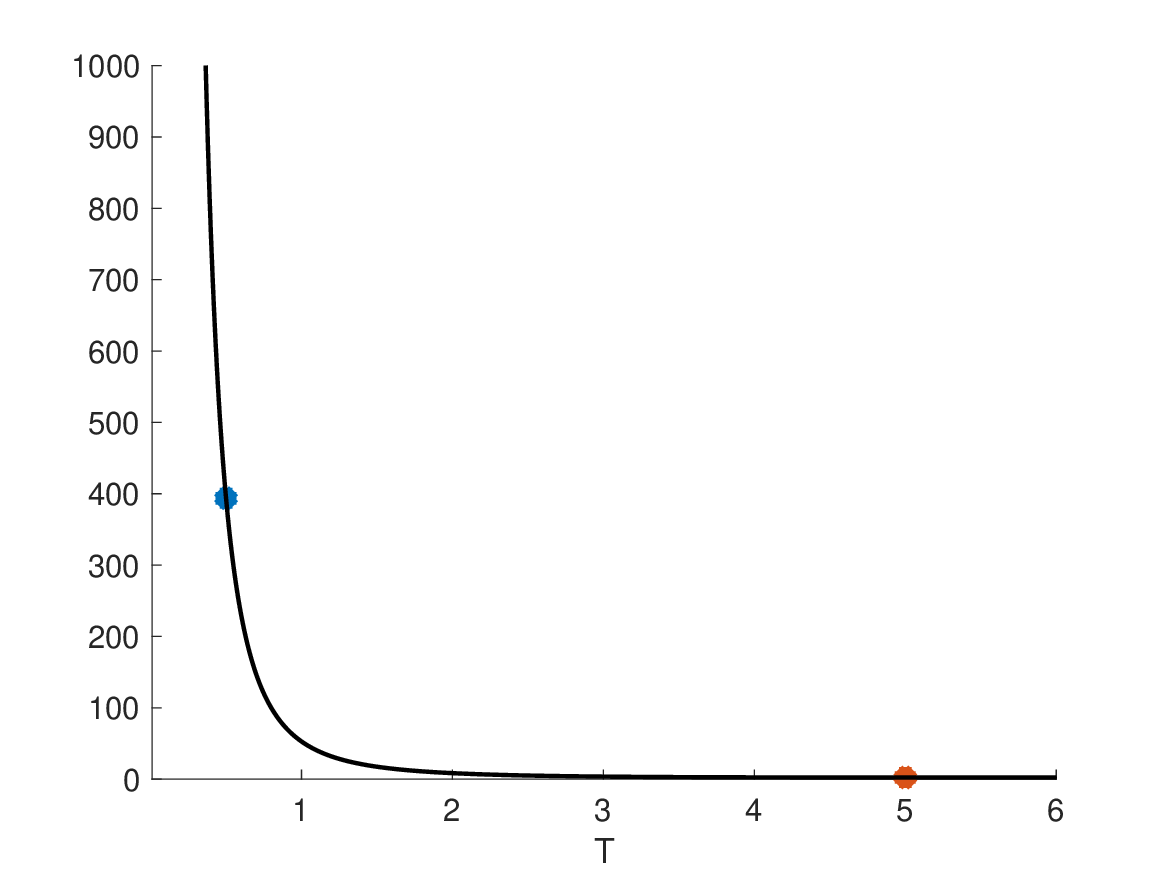,width=0.8\textwidth}}
\caption{Evolution of the criterion w.r.t. $T$}\label{J}
\end{figure*}

\begin{figure*}[!ht]
\centering
\subfigure[\footnotesize $u^\star$]
{\epsfig{figure=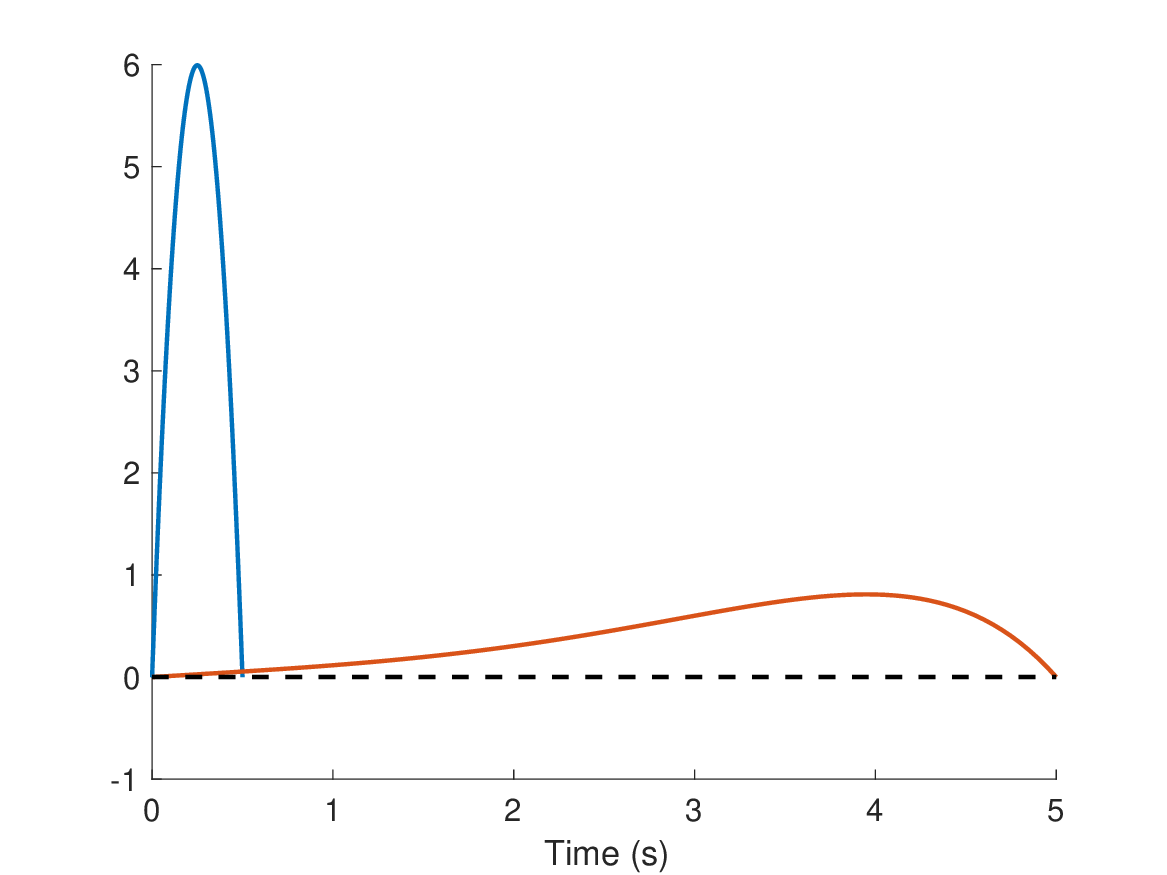,width=0.8\textwidth}}
\\
\subfigure[\footnotesize $y^\star$]
{\epsfig{figure=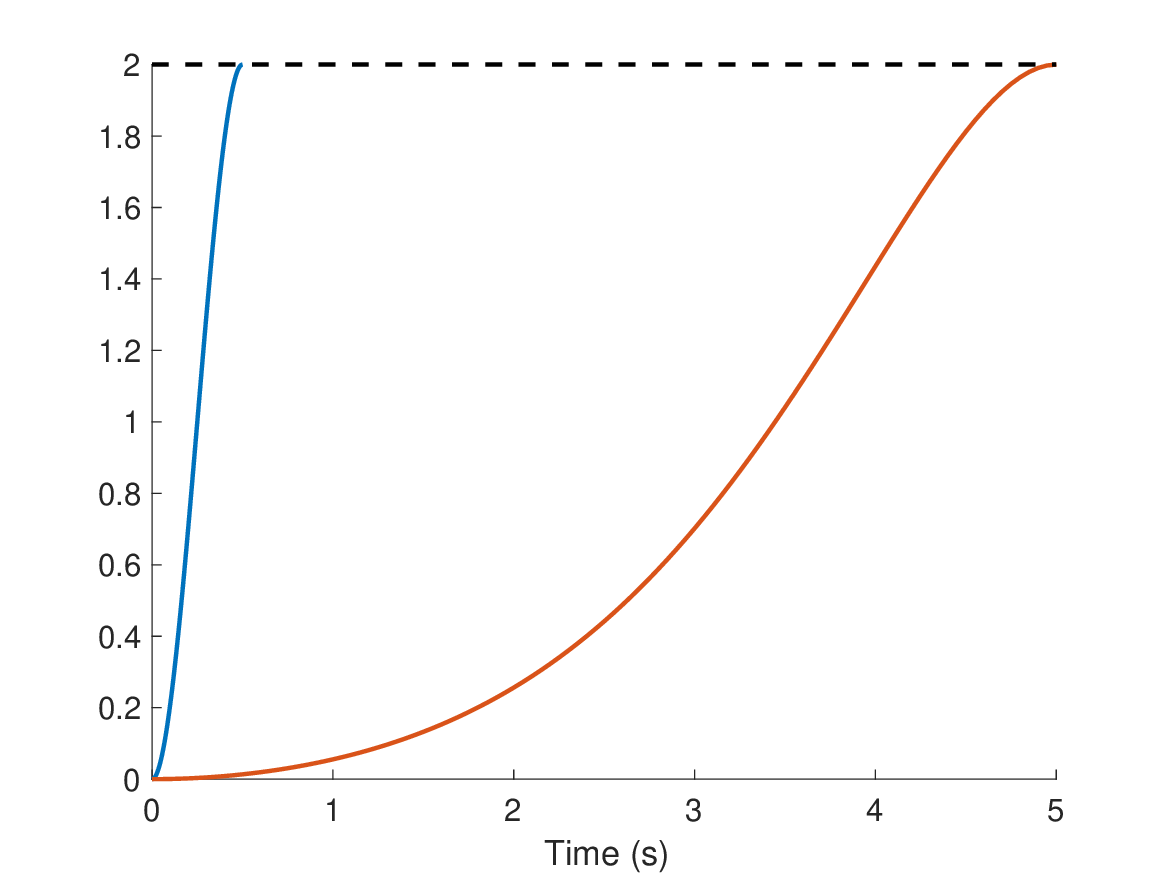,width=0.8\textwidth}}
\caption{Time-evolution of input and output w.r.t. $T=0.5$ (blue line) and $T=5$ (red line)}\label{OI}
\end{figure*}

\subsubsection{A turnpike-like phenomenon}
Consider \cite{trelat23} the double integrator $\ddot{y} = u$ with the criterion
$$J=\int_{0}^{T}\left( (u(\tau))^2+(\dot y(\tau))^2 \right) d \tau$$
The corresponding Euler-Lagrange equation is $\ddot{y} - y^{(4)} = 0$. Its general solution reads
$$y_{\rm opt} (t) = C_1 + \frac{C_2t}{2} + C_3e^t + C_4e^{-t}, \quad C_1, C_2, C_3, C_4 \in \mathbb{R}$$
Impose the boundary conditions $y(0)=0$, $y(T)={\rm y}_{1}^{T}$, $\dot{y}(0)=1$, $\dot y(T)=2$. It yields
$$
C_1=\frac{T-\mathrm{y_{1}^{T}}+{\mathrm{e}}^{2\,T}-4\,T\,{\mathrm{e}}^T+2\,\mathrm{y_{1}^{T}}\,{\mathrm{e}}^T+T\,{\mathrm{e}}^{2\,T}-\mathrm{y_{1}^{T}}\,{\mathrm{e}}^{2\,T}-1}{\left({\mathrm{e}}^T-1\right)\,\left(T-2\,{\mathrm{e}}^T+T\,{\mathrm{e}}^T+2\right)}
$$
$$
C_2= \frac{2\,\mathrm{y_{1}^{T}}-6\,{\mathrm{e}}^T+2\,\mathrm{y_{1}^{T}}\,{\mathrm{e}}^T+6}{T-2\,{\mathrm{e}}^T+T\,{\mathrm{e}}^T+2}
$$
$$
C_3=-\frac{T-\mathrm{y_{1}^{T}}+{\mathrm{e}}^T-2\,T\,{\mathrm{e}}^T+\mathrm{y_{1}^{T}}\,{\mathrm{e}}^T-1}{\left({\mathrm{e}}^T-1\right)\,\left(T-2\,{\mathrm{e}}^T+T\,{\mathrm{e}}^T+2\right)}
$$
$$
C_4=\frac{{\mathrm{e}}^T\,\left(2\,T-\mathrm{y_{1}^{T}}-{\mathrm{e}}^T-T\,{\mathrm{e}}^T+\mathrm{y_{1}^{T}}\,{\mathrm{e}}^T+1\right)}{\left({\mathrm{e}}^T-1\right)\,\left(T-2\,{\mathrm{e}}^T+T\,{\mathrm{e}}^T+2\right)}
$$
Let us switch to the  language of nonstandard analysis. Assume now that $T$ is {\em infinitely large} and ${\rm y}_{1}^{T}$ \emph{limited}, i.e., not infinitely large (see, e.g., \cite{diener,goldblatt}). Then $C_1$, $C_2$, $C_3$, $C_4$ are respectively in the {\em halo}, or {\em monad}, of $1$, $0$, $0$, $-1$, i.e., {\em infinitely close} to those quantities \cite{diener,goldblatt}. It implies that for any limited positive number $\mu$, there exists a limited positive number $\nu$ such that, for any limited time $t > \nu$, $\vert y_{\rm opt} (t) - 1\vert < \mu$. Use now a classic, but imprecise, language: $y_{\rm opt} (t)$  is close to $1$, if $t > 0$ is neither too small nor too large. Why not say that this is also a {\em turnpike} phenomenon (compare with \cite{trelat23}).

\begin{remark}
The above analysis shows at once that the three criteria, which are related to the integrator in Sect. \ref{subrest}, also exhibit turnpike 
phenomena.
\end{remark}

\begin{remark}
See  \cite{turnpike} for another approach. 
\end{remark}

\section{Closing the loop}\label{closing}

\subsection{Pole placement}\label{pole}
The Kalman state-variable representation~\eqref{kalman} yields for the variational system (see Sect. \ref{variat})
\begin{equation}\label{kalmanvar}
    \frac{d}{dt} \begin{pmatrix}
       \Delta x_1 \\ \vdots \\ \Delta x_n
    \end{pmatrix} 
    = F \begin{pmatrix}
       \Delta x_1 \\ \vdots \\ \Delta x_n
    \end{pmatrix} + G 
    \begin{pmatrix}
       \Delta u_1 \\ \vdots \\ \Delta u_m
    \end{pmatrix}
\end{equation}

\noindent Assume that Eq.~\eqref{kalman} is controllable and well-formed. Then Eq.~\eqref{kalmanvar} is also controllable and well-formed. Set $k = \mathbb{R}$. Introduce the classic state feedback (see, e.g., \cite{kailath,lunze,sontag,wolovich}), where $K \in {\mathbb{R}}^{m \times n}$,
\begin{equation}\label{feedbackstab}
    \begin{pmatrix}
       \Delta u_1 \\ \vdots \\ \Delta u_m
    \end{pmatrix} =
    - K \begin{pmatrix}
       \Delta x_1 \\ \vdots \\ \Delta x_n
    \end{pmatrix} 
\end{equation}
Combining Eqs.~\eqref{kalmanvar} and~\eqref{feedbackstab} yields 
\begin{equation*}\label{eigen}
    \frac{d}{dt} \begin{pmatrix}
       \Delta x_1 \\ \vdots \\ \Delta x_n
    \end{pmatrix} 
    = (F - GK) \begin{pmatrix}
       \Delta x_1 \\ \vdots \\ \Delta x_n
    \end{pmatrix}
\end{equation*}
The real parts of the eigenvalues of the square matrix $F - GK \in {\mathbb{R}}^{n \times n}$ may be chosen to be strictly negative. We thus obtain a stabilizing tracking feedback:
\begin{itemize}
\item the trajectory $\{x^{\star}_{\iota} (t), u^{\star}_{\kappa} (t)\}$ is stationary, i.e., it satisfies the Euler-Lagrange system of differential equations,
    \item $\Delta x_\iota (t) = x_\iota (t) - x^{\star}_{\iota} (t)$, $\iota = 1, \dots, n$;
    \item $\Delta u_\kappa (t) = u_\kappa (t) - u^{\star}_{\kappa} (t)$, $\kappa = 1, \dots, m$. 
\end{itemize}

\subsection{Homeostat}\label{stat}
\subsubsection{The monovariable case}
Consider a controllable dynamics $\Lambda$ with a single control variable $u$. Let $y$ be a flat output. It yields for the variational system $\Lambda_\Delta$
\begin{equation}
\label{mono}
\Delta u = \sum_{\rm finite} a_\varepsilon \frac{d^\varepsilon}{dt^\varepsilon} (\Delta y), 
\quad a_\varepsilon \in k
\end{equation}
Call $\nu$, $\nu \geqslant 1$, the least integer such that $a_\nu \neq 0$. Eq.~\eqref{mono} may be rewritten as
\begin{equation}
\label{interm}
\frac{d^\nu}{dt^\nu} (\Delta y) =  \mathfrak{F} + \alpha \Delta u   
\end{equation}
where $\mathfrak{F} = \frac{- 1}{a_\nu} \sum_{\varepsilon \neq \nu} a_\varepsilon \frac{d^\varepsilon}{dt^\varepsilon} (\Delta y)$, $\alpha = \frac{1}{a_\nu}$.

When $k = \mathbb{R}$, the \emph{homeostat}, which is replacing the ultra-local model \cite{mfc1,mfc2}, is deduced from Eq.~~\eqref{interm}:
\begin{equation}
\label{driver}
\frac{d^\nu}{dt^\nu} (\Delta y)  =  F + \alpha \Delta u   
\end{equation}
There
\begin{itemize}
    \item $\Delta y = y - y^\star$, $\Delta u = u - u^\star$, where $y^\star$,$u^\star$ are stationary (see Sect. \ref{EL});
    \item $F = \mathfrak{F} + G$, where according to the model-free framework \cite{mfc1,mfc2}, $G$ stands for all the mismatches and disturbances.
\end{itemize}


\subsubsection{Some data-driven calculations}
The following calculations cover the most common cases, i.e., $\nu = 1, 2$ in Eq.~\eqref{driver} (see \cite{mfc1,mfc2,heol}).

For $\nu = 1$, the following estimate of $F$ has been obtained \cite{mfc1,mfc2,heol} via usual operational calculus:
\begin{equation*}
\label{estim1} F_{\rm est} = - \frac{6}{\mathfrak{t}^3} \int_{0}^{\mathfrak{t}} \left( (T - 2 \sigma)\Delta \tilde{y}(\sigma) {+} \alpha \sigma (T - \sigma)\Delta \tilde{u}(\sigma)\right)d\sigma
\end{equation*}
where 
\begin{itemize}
\item the time lapse $\mathfrak{t} > 0$  is ``small,''
\item $\Delta\tilde y(\sigma)=\Delta y(\sigma+t-\mathfrak{t})$, $\Delta \tilde u(\sigma)= \Delta u(\sigma+t-\mathfrak{t})$.
\end{itemize}

For $\nu = 2$ , 
\begin{eqnarray*}\label{2bis}
    F_{\rm est} &=& \dfrac{60}{\mathfrak{t}^5}\left[
            \int_0^\mathfrak{t} \left(\big(\mathfrak{t}-\sigma\big)^2 -4\big(\mathfrak{t}-\sigma\big)\sigma + \sigma^2\right)\Delta \tilde{y}(\sigma)d\sigma
            \right.\nonumber \\
    &&      \mbox{} -
       \left.\dfrac{\alpha}2 \int_0^\mathfrak{t} (\mathfrak{t}-\sigma)^2\sigma^2\Delta  \tilde{u}(\sigma)d\sigma\right]
  \end{eqnarray*}

\subsubsection{Intelligent controllers}\label{intelcontr}

Introduce \cite{mfc1}, when $\nu = 1$, the \emph{intelligent proportional} controller, or \emph{iP},
\begin{equation}
\label{ip}
\Delta u = - \frac{F_{\text{est}} + K_P \Delta y }{\alpha}    
\end{equation}
where $K_P \in \mathbb{R}$ is the \emph{gain}. Combine Eqs.~\eqref{driver} and~\eqref{ip}:
$$
\frac{d}{dt} (\Delta y) + K_P \Delta y = F - F_{\text{est}}
$$
If 
\begin{itemize}
    \item the estimate of $F$ is ``good'', \textit{i.e.}, $F - F_{\text{est}} \approx 0$,
    \item $K_P >0$, 
\end{itemize}
then $\displaystyle\lim_{t \to +\infty} \Delta y \approx 0$. This local stability result is easily extended \cite{mfc1,mfc2} to the case $\nu = 2$ via the \emph{intelligent proportional-derivative} controller, or \emph{iPD},
\begin{equation}
\label{ipd}
\Delta u = - \frac{F_{\text{est}} + K_P \Delta y + K_D \frac{d}{dt} (\Delta y)}{\alpha}
\end{equation}
where the gains $K_P, K_D \in \mathbb{R}$ are chosen such that the roots of $s^2 + K_D s + K_P$ have strictly negative real parts. 
\begin{remark}
Set $\Delta Y = \Delta y + \int_{c}^{t} \Delta y(\sigma) d\sigma$, $0 \leqslant c < t$, and $\mathcal{F} = F + K_D \frac{d}{dt}(\Delta y)$. It yields 
\begin{equation}\label{22}
\frac{d^2}{dt^2}(\Delta Y) = \mathcal{F} + \alpha u
\end{equation}
and
\begin{equation}\label{riachy}
\Delta u = - \frac{\mathcal{F}_{\text{est}} + K_P \Delta y}{\alpha} 
\end{equation}
where $\mathcal{F}_{\text{est}}$ is obtained by adapting Eq.~\eqref{2bis} via Eq.~\eqref{22}. Eq.~\eqref{riachy} permits avoiding the estimation of the derivative $\Delta y$: This is {\em Riachy's trick} \cite{mfc2}.
\end{remark}

\subsubsection{An illustrative example}
Consider again Eq. \eqref{ex3}. Take as reference trajectory the time-optimal path in Sect. \ref{oth}. The corresponding homeostat is   
$\frac{d}{dt}(\Delta y) = F+\alpha \Delta u$, $y=x_1$, $\alpha=3/4$.
In order to test the robustness with respect to model mismatches and disturbances,
\begin{itemize}
    \item replace $2x_1$ in Eq. \eqref{ex3} by $x_1$,
    \item set $\Delta y (0) = 10$,
    \item introduce an additive white-Gaussian corrupting noise 
    $\mathcal{N}(0,0.5)$.
\end{itemize}
Fig. \ref{BF2} exhibits an excellent tracking of the output reference trajectory.

\subsubsection{The multivariable case}
Take a controllable and well-formed dynamics $\Lambda$ with $m$ independent control variables ${\bf u} = \{u_1, \dots, u_m\}$. Let ${\bf y} = \{y_1, \dots, y_m\}$ be a flat output. Then Eq.~\eqref{mono} is replaced, $\forall j = 1, \dots, m$, by
\begin{equation}
\label{multiv}
    \sum_{\rm finite} c_{j,\varsigma,\varpi} \frac{d^\varpi}{dt^\varpi} (\Delta y_\varsigma) = \Delta u_j, \qquad c_{j,\varsigma,\varpi} \in k
\end{equation}
It can be assumed that, after an appropriate renumbering of the flat output components, $\Delta y_j$ appears in Eq.~\eqref{multiv}. Let $\varpi_j \geqslant 1$ be the smallest integer such that $c_{j,j,\varpi_j} \neq 0$. It yields
\begin{equation}
\label{interm:multi}
\frac{d^{\varpi_j}}{dt^{\varpi_j}} (\Delta y_j) =  \mathfrak{F}_j + {\mathfrak{a}_j} \Delta u_j, \qquad    \mathfrak{a}_j \in k
\end{equation}
Set $k = \mathbb{R}$. The \emph{homeostat}, which is deduced from Eq.~~\eqref{interm:multi}, reads
\begin{equation*}
\label{driver1}
{\frac{d^{\mu_j}}{dt^{\mu_j}}}(\Delta y_j)  =  F_j + \alpha_j \Delta u_j,\qquad j=1,\ldots,m
\end{equation*}
with
\begin{itemize}
    \item $\Delta y_j = y_j - y^\star_j$, $\Delta u_j = u_j - u^\star_j$, where $Y^\star = \{y_1^\star,\ldots,y_m^\star\}$, $U^\star = \{u_1^\star,\ldots,u_m^\star\}$ is an extremum;
    \item $F_j = \mathfrak{F}_j + G_j$, where $G_j$ stand for all the mismatches and disturbances.
\end{itemize}
The extension of Section~\ref{intelcontr} to the multivariable case is straightforward. 

\section{Conclusion}\label{conclusion}
Some of our results were sketched in \cite{batna}. This short research announcement furthermore indicates possible nonlinear extensions. They will be developed elsewhere. Note also that our new viewpoint on optimal control, combined with model-free control, permits a clear-cut definition of \emph{model-free predictive control} \cite{mfpc,maroc}, which is currently a burning topic (see numerous references in \cite{mfpc}).

Module theory was proposed by Kalman \cite{pnas} a long time ago and in a completely different manner, without, it must be said, leaving much of a trace. It explains why those tools are currently unknown to most control theorists, despite a few publications (see, e.g., in chronological order, \cite{sain}, \cite{lomadze}, \cite{oberst}, and other references in Remark \ref{rem}). Our approach to linear quadratic regulators, which differs from the Kalman LQR in all kinds of key aspects, such as the two-point boundary problem and the optimal time horizon, will hopefully inspire theoretical, numerical, and concrete technological advances. It might finally win acceptance for \emph{modern algebra} in the sense of 
van der Waerden's celebrated book \cite{wa}, which is based on lectures given by Emil Artin and Emmy Noether in 1926 and 1929, one century ago.

\begin{figure}[H]
\centering
\subfigure[\footnotesize $u^\star$ (blue - -) $u$ (red --)]
{\epsfig{figure=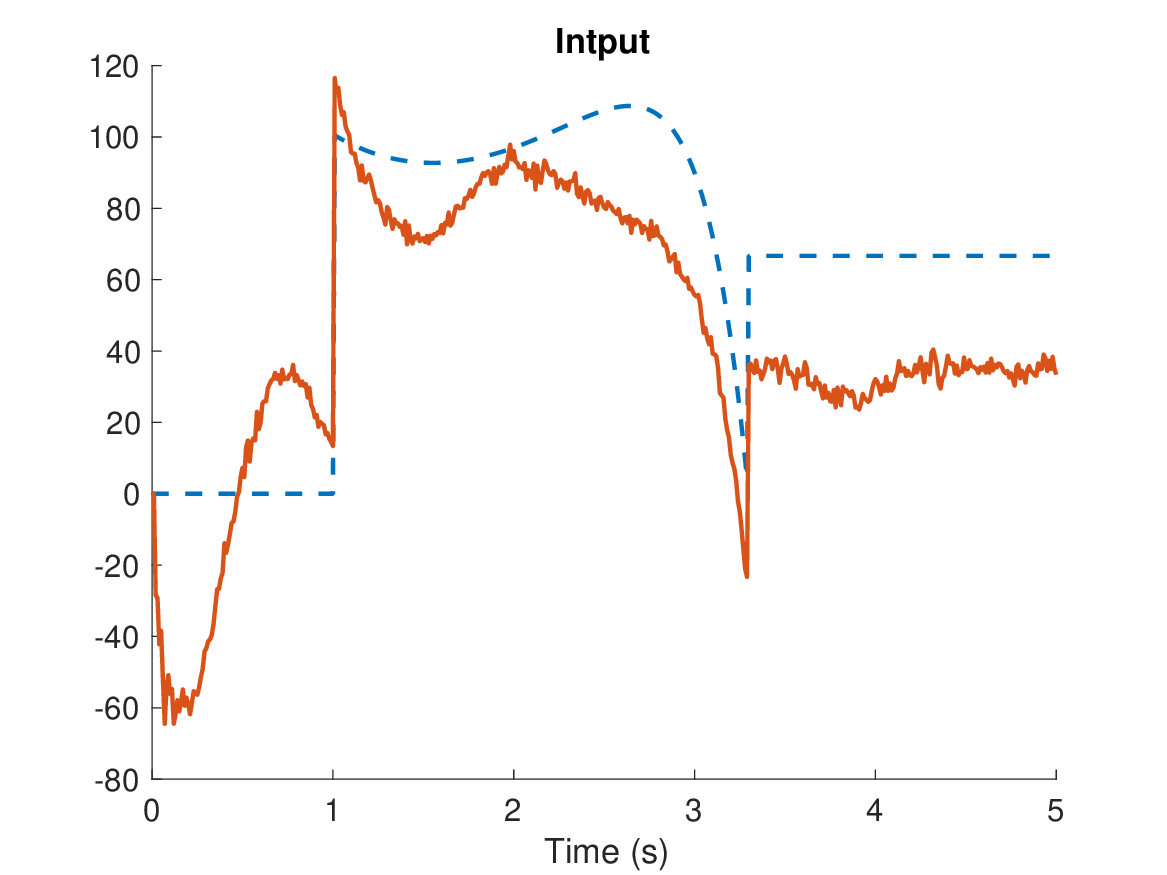,width=0.8\textwidth}}
\\
\subfigure[\footnotesize $y^\star$ (blue - -) $y$ (red --)]
{\epsfig{figure=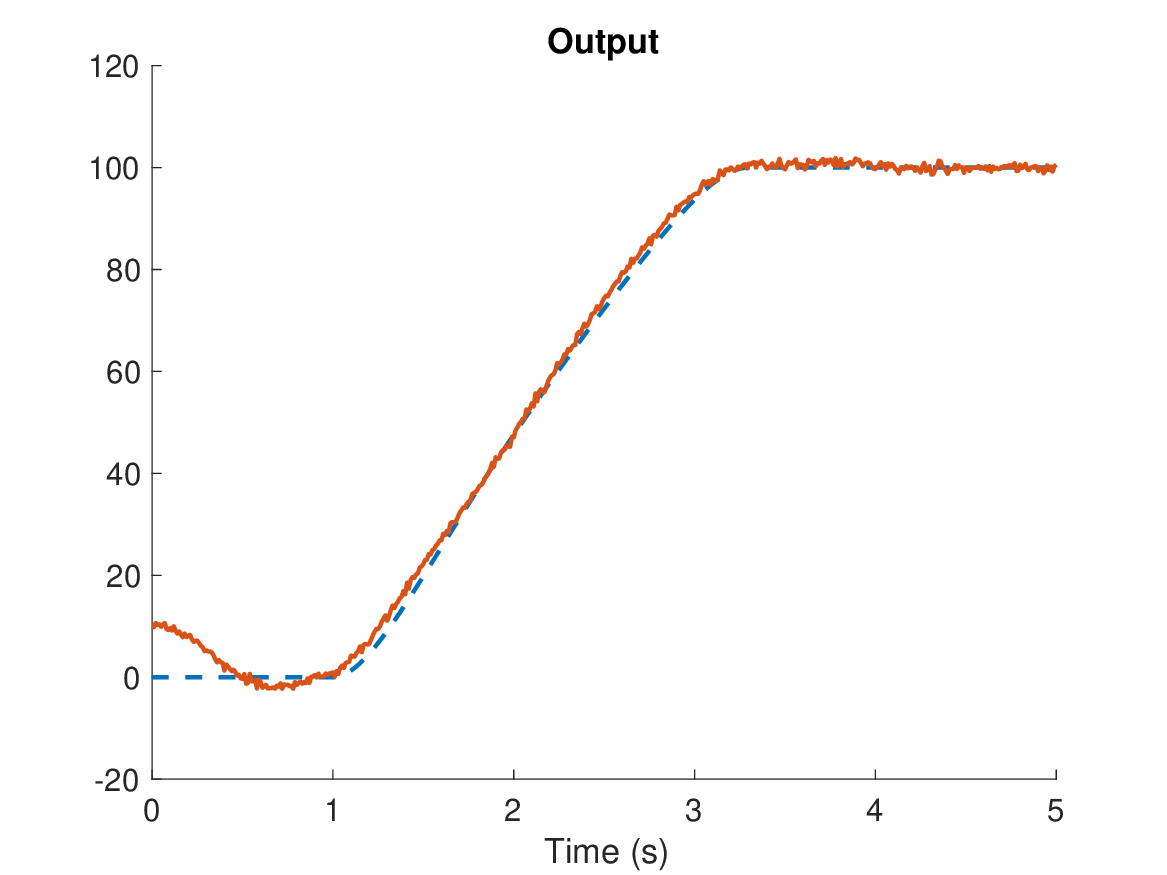,width=0.8\textwidth}}
\caption{Closed loop control: uncertain case}\label{BF2}
\end{figure}



\end{document}